\documentclass[11pt]{amsart}

\usepackage{amsthm,amsmath,amssymb,amscd,amsfonts}
\usepackage{mathrsfs}
\theoremstyle{plain}
\newtheorem{theorem}{Theorem}[section]

\newtheorem{proposition}[theorem]{Proposition}

\newtheorem{corollary}[theorem]{Corollary}
\newtheorem{def-thm}[theorem]{Definition-Theorem}
\newtheorem{lemma}[theorem]{Lemma}

\newtheorem{definition}[theorem]{Definition}

\theoremstyle{definition}

\newcommand{\sq}[1]{\ifx#1([\else\ifx#1)]%
  \else\message{invalid use of "sq"}\fi\fi}

\DeclareMathSymbol{\idot}{\mathbin}{operators}{`\.}
\allowdisplaybreaks
\begin{document}
\title{ The Second Main Theorem in the hyperbolic case}

\author{Min Ru }
\address{Department of Mathematics\\University of Houston\\4800 Calhoun Road, Houston, TX 77204\\USA.}
\email{minru@math.uh.edu}

\author{Nessim Sibony}
\address{Department of Mathematics\\Laboratoire de Math\'ematiques d'Orsay, Univ. Paris-Sud, CNRS, Universit\'e
Paris-Saclay,\\ 91405 Orsay, France.}
\email{Nessim.Sibony@math.u-psud.fr}

\subjclass[2000]{}

\begin{abstract}  We develop   Nevanlinna's  theory for a class of holomorphic maps when the source  is a disc. Such maps appear in the theory of foliations by Riemann Surfaces.
\end{abstract}
\thanks{The first named author was supported in part by the Simon Foundation Awd\# 527078}
\maketitle

\section{Introduction}\label{intro}
In 1929, Nevanlinna \cite{Nev} established the Second Main Theorem for meromorphic functions on the complex plane ${\Bbb C}$.
Later, S. S. Chern \cite{Chern} extended the result to holomorphic mappings from the complex plane
 into  compact Riemann surfaces.  In 1933, H. Cartan \cite{Cartan} developed the theory for 
holomorphic mappings from the complex plane to ${\Bbb P}^n({\Bbb C})$ and studied the intersection with  hyperplanes in 
general position. At the same time, it was observed (first by Nevanlinna) that the results  also hold for 
meromorphic functions  on the unit disc $\bigtriangleup(1)$, under the condition that 
$$\lim_{r \rightarrow 1}{ T_f(r)\over \log {1\over 1-r}}=\infty.$$
Tsuji \cite{Tsuji} gives an exposition of this theory.
In this paper, we introduce a new class of maps from the disc of radius 
$R$  with $0<R\leq \infty$, for which we obtain a Second Main Theorem.  
 Let 
$\bigtriangleup(R)$ denote the  disc of radius $R$ with the 
convention that $\bigtriangleup(\infty)={\Bbb C}$.
 Let $M$ be a Hermitian  manifold and $\omega$ be a positive $(1,1)$ form of finite mass on $M$.
 Recall that, for a non-constant holomorphic map  $f: \bigtriangleup(R) \rightarrow M$, 
 the {\it characteristic (or height) function of $f$ with respect to $\omega$} is defined, for $0<r<R$, as
$$T_{f, \omega}(r)=\int_0^r {dt\over t} \int_{|z|<t} f^*\omega.$$
For each $c<\infty$, let 
\begin{equation}\label{cf}
{\mathcal E}_c=\left\{f~\big|~~\int_0^R \exp(cT_{f, \omega}(r))dr=\infty \right\},
\end{equation}
\begin{equation}
\mathcal E= \cup_{c<\infty} {\mathcal E}_c ~~~\text{and}~~~ {\mathcal E}_0= \cap_{c>0} {\mathcal E}_c .
\end{equation}
Observe that the set  ${\mathcal E}_c$ contains  the maps from the unit disc to $M$ which satisfy, for $r$ close to $1$,
$${T_{f, \omega}(r)\over \log{1\over 1-r}}\ge {1\over c}.$$
This is an important class of maps.  They occur, for example, 
as the universal covering maps of leaves in foliation
by Riemann surfaces.  This is our main motivation. 

Generic foliations in ${\Bbb P}^n({\Bbb C})$ are 
``Brody hyperbolic",  i.e. they do not admit a non-constant image of ${\Bbb C}$ tangent to the foliation
out of the singular points (see \cite{Br} and  \cite{LN}). So leaves are uniformized by the unit disc. It turns out that frequently the uniformizing map is in ${\mathcal E}_c$ . When the  foliation is ``Brody hyperbolic'', we get also that
${T_{f, \omega}(r)\over \log{1\over 1-r}}$ is bounded.

 It is conjectured that, for generic foliations, the leaves are dense. So their distribution is far from trivial. The extension of Cartan's theorem which we obtain can be applied.
 In the next section, we will list some examples. The space $ {\mathcal E}_0$ is the space of maps of fast growth.

\begin{definition}
 Let $M$ be a complex manifold and $\omega$ be a positive $(1,1)$ form of finite volume on 
 $M$. Let $0<R\leq \infty$ and $f: \bigtriangleup(R) \rightarrow M$ be a holomorphic map. 
We define the {\bf growth index} of $f$ with respect to $\omega$ as  
\begin{equation} \label{c1}
c_{f, \omega}: =\inf\left\{c>0~\big|~\int_0^R \exp(cT_{f, \omega}(r))dr=\infty\right\}.
\end{equation}
  The  
critical constant of $M$ with respect to $\omega$, denoted by $c_{cri, M}^{\omega}$, is defined as
$$c_{cri, M}^{\omega} =\inf\{c~|~\exists \text{a non-constant 
holomorphic map}~f: \bigtriangleup(1)\rightarrow M, $$
$$\text{and} ~\int_0^1 \exp(cT_{f, \omega}(r))dr=\infty\}.$$

\end{definition}
In this paper,  whenever $c_{f,\omega}$ is involved, 
we always assume that the set $\left\{c>0~|~\int_0^R \exp(cT_{f, \omega}(r))dr=\infty\right\}$ is non-empty. 
If $f$ is of bounded characteristic (hence $R<\infty$), then $c_{f, \omega}=\infty$.
 In the case where $R=\infty$, noticing that $\int_0^R \exp(\epsilon T_{f, \omega}(r))dr=\infty$ 
for any arbitrary small $\epsilon$ if $f$ is not constant, we get that $c_{f, \omega}=0$ and $f$ is in 
${\mathcal E}_0.$
Thus our results also include the classical 
results for mappings  on the whole complex plane $f:  {\Bbb C}\rightarrow M$.

When $M$ is compact the spaces  $ {\mathcal E}$ and  $ {\mathcal E}_0$ are independent of the form $\omega,$ so they are intrinsic objects.
Indeed we can characterize the Kobayashi hyperbolicity by using   $ {\mathcal E}_0$ (see Theorem \ref{Ko} below) as follows: 
{\it  Let $M$ be a compact complex manifold.  Then $M$ is hyperbolic if and only if 
  the class $\mathcal E_0(\bigtriangleup(1))$ is empty }. 

The Second Main Theorems will be derived for  maps  $f: \bigtriangleup(R) \rightarrow M$ with 
$c_{f,\omega}<\infty$. In particular, we
derive the defect for $f$ in $M$ in terms of $c_{f, \omega}$.
  In the  case where $M$ is hyperbolic, for example $M$ is a Riemann surface of genus 
$\ge 2$, there is no non-constant holomorphic map
$f: {\Bbb C}\rightarrow M$. However, there are many non-constant  holomorphic maps 
$f: \bigtriangleup(1)\rightarrow M$ which are in $\mathcal E$. Our result (see Theorem \ref{g}) shows that if $c_{f, \omega_P}<\infty$, then 
$\sum_{j=1}^ q \delta_{f, \omega_P}(a_j)\leq c_{f, \omega_P}-1$ for any distinct points 
 $a_1, \dots, a_q\in M$.
Here $\omega_P$ is the Poincar\'e form on $M$ and $\delta_{f, \omega_P}(a)$ is 
the defect properly measured.  This is a  new  phenomenon. We also get a similar result  for 
a compact Riemann surface with finitely many points removed.

The theory  here can be regarded as  a new illustration of Bloch's principle: 
 {\it Nihil est in infinito quod non prius fuerit in finito}. This is explained as: every proposition with a statement  on the actual infinity  can be always considered a consequence of a proposition in finite terms.

We introduce some notations. 
For a complex variable $z$,  let 
$$\partial u = {\partial u \over \partial z} dz, ~~
\bar{\partial} u = {\partial u \over \partial \bar{z}} d\bar{z}.$$
Let $d= \partial +\bar{\partial}, d^c={\sqrt{-1}\over 4\pi} (\bar{\partial}-\partial)$.
We have
 $dd^c={\sqrt{-1}\over 2\pi} \partial \bar{\partial}$.
Let $M$ be a Riemann surface. 
 Let $\omega = a(z)  {\sqrt{-1}\over 2\pi} dz\wedge d{\bar z}$
be a  non-negative $(1,1)$ form on $M$.
Let $\mbox{Ric} (\omega):=dd^c\log a$.  Then we have 
$$\mbox{Ric} (\omega) = - K \omega,$$
where $K$ is the Gauss curvature of the metric form $\omega$. 
For example, on the unit  disc $\bigtriangleup(1)$, 
the Poincar\'e  metric  form $\omega = {2\over
(1- |z|^2)^2} {\sqrt{-1}\over 2\pi} dz \wedge d{\bar z}$
has Gauss curvature $-1$.

We state our results. For notations, see Section 2.
\begin{theorem}[The Second Main Theorem]\label{tha}
Let $M$ be a compact Riemann surface.
Let $\omega$ be a  smooth positive (1,1) form on $M$. 
Let 
$f:  \bigtriangleup(R)\rightarrow  M$ be a 
holomorphic map with $c_{f, \omega}<+\infty$, where $0<R\leq \infty$. 
Let $a_1, \dots, a_q$ be distinct points on $M$. 
  Then, for
 every  $\epsilon > 0$, the inequality
\begin{eqnarray*}
&~& \sum_{j=1}^q m_{f, \omega}(r, a_j) + T_{f, \mbox{Ric}(\omega)}(r)  + N_{f, \mbox{ram}}(r) \\
&~&\leq   (1+\epsilon) (c_{f, \omega}+\epsilon) T_{f, \omega}(r) + O(\log  T_{f, \omega}(r))
+\epsilon  \log r
\end{eqnarray*} holds for all $r \in (0, R)$  outside a set $E$ 
with 
$\int_E \exp((c_{f, \omega}+\epsilon) T_{f, \omega}(r))dr <\infty$. Here $N_{f, \mbox{ram}}(r)$ is the counting function for the ramification divisor of $f$.
\end{theorem}

\noindent{\bf Remarks}. (a) We note that in the case where $R= \infty$ we have $c_{f, \omega}=0$, so we recover the usual
Second Main Theorem for $f:  {\Bbb C} \rightarrow  M$ (due to Chern) with a better error term: $\epsilon  \log r$.
The error term is  $\epsilon \log r$ rather than $O(\log r)$, so  we don't need anymore to assume that  $f$ is transcendental.

(b) The above theorem also holds for an open set  $U$ in a compact Riemann surface $M$ such that $M\backslash U$ is a set of finite number of points. To get positive results, we need to consider a metric defined only in $U$. See the remark after Theorem \ref{1}.

(c) Note that we can also let $c_{f, \omega}$ depend on $r$, i.e., we can consider the $c(r)>0$
with $$\int_0^R \exp(c(r) T_{f, \omega}(r))dr=\infty.$$
We then get similar results.

In the case where  $M={\Bbb P}^1({\Bbb C})$,
since $$\omega_{FS} = {1\over (1+|w|^2)^2} {\sqrt{-1}\over 2\pi} dw
\wedge d{\bar w} = dd^c \log (1+|w|^2),$$
we get that $$\text{Ric} (\omega_{FS}) = -2  \omega_{FS}.$$
Hence Theorem \ref{tha} gives 
\begin{theorem}\label{1}  Let
 $f: \bigtriangleup(R) \rightarrow {\Bbb P}^1({\Bbb C})$ be a 
 holomorphic map such that $c_f<+\infty$, where $c_f:=c_{f, \omega_{FS}}$ and  $0<R\leq \infty$. Let $a_1, \dots, a_q$ be distinct points on ${\Bbb P}^1({\Bbb C})$. Then 
\begin{equation}\label{defect}\sum_{j=1}^q \delta_f(a_j)\leq 2+ c_f.\end{equation}
 In particular, $f$ cannot omit more than $[2+c_f]$ points in  ${\Bbb P}^1({\Bbb C})$ if $c_f$ is finite.
\end{theorem}

\noindent{\bf Remark}. Let $U$ be an open subset of $
{\Bbb P}^1({\Bbb C})$ such that ${\Bbb P}^1({\Bbb C})\backslash U$ is an infinite set. 
Let $\phi$ denote the universal covering map $\phi: \bigtriangleup(1) \rightarrow U$. 
From the fact that the image of $\phi$ omits infinitely many points in $ {\Bbb P}^1({\Bbb C})$, 
Theorem \ref{1} tells us that $c_f=\infty$. If ${\Bbb P}^1({\Bbb C})\backslash U$ is finite, then 
Theorem \ref{1} implies that 
 $c_f \ge (q-2)$ where $q=\#({\Bbb P}^1({\Bbb C})\backslash U)$. 

In the 
 elliptic case, the canonical metric 
is flat , i.e. there exists a positive (1,1) form $\omega$ whose curvature is 
0, so 
$\mbox{Ric}(\omega) =0$. As a consequence of Theorem \ref{tha}, we get
\begin{theorem}\label{2} Let $M$ be  a compact Riemann surface of genus 1 and 
 $\omega$ be the  positive (1,1) form with $\mbox{Ric}(\omega) =0$. 
 Let $f: \bigtriangleup(R) \rightarrow M$ be a  holomorphic map with 
$c_{f, \omega}<\infty$, where $0<R\leq \infty$.
Then $$\sum_{j=1}^q \delta_{f,\omega}(a_j) \leq c_{f, \omega}.$$
In particular,  
 $f$ cannot omit more than $[c_{f, \omega}]$ points in  $M$ if $c_{f, \omega}$ is finite.
\end{theorem}
In the case where  the compact Riemann surface is of genus $\ge 2$,
there is  a positive (1,1) form $\omega$ whose curvature $-1$ so 
$\text{Ric}(\omega)=\omega$. 
 We get the following result using a variation of the proof of Theorem \ref{tha}.
\begin{theorem}\label{g}  
 Let $U$ be either a compact Riemann surface or a Riemann surface  in a compact Riemann surface $M$ such that $M\backslash U$ consists of  a finite number of points.
Let $\omega$ be a positive  (1,1) form of finite volume on $U$ whose Gauss curvature is  bounded from 
above by $-\lambda$ with $\lambda>0$, i.e. $\mbox{Ric}(\omega) \ge \lambda\omega$. 
Let 
 $f: \bigtriangleup(R) \rightarrow  U$ be a  holomorphic map with $c_{f, \omega}<\infty$, 
where $0<R\leq \infty$. Then  $c_{f, \omega}\ge \lambda$. Furthermore,
 let $a_1, \dots, a_q$ be distinct points on $U$, then, for
 every  $\epsilon > 0$, the inequality
\begin{eqnarray*}
&~& \sum_{j=1}^q m_{f, \omega}(r, a_j)  +  N_{f, \mbox{ram}}(r) \\
&~&\leq   ((1+\epsilon)(c_{f, \omega}+\epsilon) -\lambda) T_{f, \omega}(r) + O(\log  T_{f, \omega}(r))
+\epsilon  \log r
\end{eqnarray*}
 holds for all $r \in (0, R)$  outside a set $E$ 
with 
$\int_E \exp((c_{f, \omega}+\epsilon) T_{f, \omega}(r))dr <\infty$.
In particular,  we have $$\sum_{j=1}^q \delta_{f, \omega}(a_j)\leq  c_{f, \omega}-\lambda.$$
  \end{theorem}
When $U$ is hyperbolic, 
there is no non-constant holomorphic map  $f: {\Bbb C} \rightarrow  U$.  However, there are many 
 non-constant maps from the unit-disk into $U$, for example, the universal 
covering map $\phi:  \bigtriangleup(1) \rightarrow  U$. If 
we take the Poincar\'e metric form  $\omega_P$ (i.e., whose Gauss curvature is $-1$), 
then it is easy to compute that $c_{\phi,\omega_P}=1$ since 
 $\phi^*\omega_P$ is the Poincar\'e metric on $\bigtriangleup(1)$.
On the other hand, from 
Theorem \ref{g} above, we  know  that  {\it for  any  non-constant
 holomorphic map $f: \bigtriangleup(1) \rightarrow  U$ we have  $c_{f, \omega_P}\ge  1$.}
So the  universal 
covering map $\phi:  \bigtriangleup(1) \rightarrow  U$ is the (non-constant) map whose growth index 
 achieves the lower bound  $1$. 

Part of the  above theorem can be extended to  higher dimension. 
 Theorem 5.7.2 in \cite{Vojta}, corresponds to the case $c_f=0, R=\infty$ of the following theorem.  
\begin{theorem}\label{thas}
Let $\omega$ be a  positive $(1, 1)$-form on a compact complex manifold $V$ whose 
 holomorphic sectional curvature is bounded from 
above by $-\lambda$ with $\lambda>0$, i.e.  for any holomorphic map $g: U\rightarrow V$
$($$U\subset  {\Bbb C}$ is an open subset$)$, 
 $\mbox{Ric}(g^*\omega) \ge \lambda g^*\omega$.
Let 
$f:  \bigtriangleup(R)\rightarrow  V$ be a 
holomorphic map  with  $c_{f, \omega}<\infty$, where $0<R\leq \infty$.  Then, for every $\epsilon >0$,  the inequality
$$ (\lambda- (1+\epsilon)(c_{f, \omega}+\epsilon))T_{f, \omega}(r)  +  N_{f, \mbox{ram}}(r) 
\leq  O(\log T_{f,\omega}(r)) +\epsilon \log r$$  holds for all $r \in (0, R)$  outside a set $E$ 
with 
$\int_E \exp((c_{f, \omega}+\epsilon)T_{f, \omega}(r))dr <\infty$.
In particular, we have
$$c_{f, \omega}\ge \lambda.$$
\end{theorem}

From Theorem \ref{thas}, if $M$ is  a Hermitian manifold  and 
$\omega_P$ is  a positive   (1,1) form on $M$ whose holomorphic sectional curvature is bounded from 
above by $-1$ on $M$, then  $c_{cri, M}^{\omega_P}\ge 1$.  

We now turn to the Second Main Theorem for holomorphic curves in ${\Bbb P}^n({\Bbb C})$. We prove the following theorem which 
generalizes (by taking $R=\infty$) the result of Nochka. 
\begin{theorem}\label{thb} 
Let $f:   \bigtriangleup(R) \rightarrow {\Bbb  P}^n({\Bbb C})$  be a
holomorphic map with $c_f<\infty$, where $c_f=c_{f, \omega_{FS}}$ and $0<R\leq \infty$.
Assume that the  image of $f$  is contained in some 
$k$-dimensional subspace of $ {\Bbb  P}^n({\Bbb C})$ but not in any subspace of dimension lower than 
$k$. Let  $H_j$, $1 \leq j \leq q$, be  hyperplanes in 
${\Bbb P}^n({\Bbb C})$  in general position.  
Assume that $f(\bigtriangleup(R))\not\subset H_j$ 
for $1 \leq j \leq q$. Then, for any $\epsilon > 0$, the inequality,
\begin{eqnarray*}
\lefteqn{\sum_{j=1}^q ~m_{f, H_j}(r) + \left({n+1\over k+1}\right)N_{f, ram}(r) \leq (2n-k+1)T_f(r)}\\ 
&&+{(2n -k+1)k\over 2}\left((1+\epsilon)(c_f+\epsilon) T_f(r) +\epsilon \log r\right) +  O(\log T_f(r))
\end{eqnarray*}
holds  for all $r \in (0, R)$  outside a set $E$ 
with 
$\int_E \exp((c_f+\epsilon)T_{f}(r))dr <\infty$. Here $N_{f, ram}(r)$ is the counting function for the ramification divisor of $f$.
\end{theorem}
When $k=n$, this gives an  extension of H. Cartan's result.
\begin{corollary}
 Let $H_1,\dots,H_q$ be hyperplanes in 
${\Bbb P}^n({\Bbb C})$ in general position. Let 
$f:   \bigtriangleup(R) \rightarrow {\Bbb  P}^n({\Bbb C})$ be a linearly non-degenerate
 holomorphic curve (i.e. its image is not
 contained in any proper subspace of $ {\Bbb  P}^n({\Bbb C})$) with $c_f<\infty$, where $c_f=c_{f, \omega_{FS}}$ and $0<R\leq \infty$.
   Then, for any $\epsilon > 0$, the  inequality
\begin{eqnarray*}
&~&\sum_{j=1}^q m_f(r, H_j) + N_{W}(r, 0)\leq (n+1)T_f(r) + {n(n+1)\over 2} (1+\epsilon)(c_f+\epsilon) T_f(r)\\
&~& + O(\log T_f(r)) +{n(n+1)\over 2}\epsilon \log r
\end{eqnarray*}
holds  for all $r \in (0, R)$  outside a set $E$ 
with 
$\int_E \exp((c_f+\epsilon)T_{f}(r))dr <\infty$.
Here  $W$ denotes the Wronskian of $f$. 
\end{corollary}
As a consequence of Theorem \ref{thb}, we get
\begin{corollary}  Let $H_1,\dots,H_q$ be hyperplanes in 
${\Bbb P}^n({\Bbb C})$ in general position. Let 
$f:   \bigtriangleup(R) \rightarrow {\Bbb  P}^n({\Bbb C})$ be a 
non-constant  holomorphic curve
  with $c_f<\infty$, where $c_f=c_{f, \omega_{FS}}$ and $0<R\leq \infty$.
  Assume that $f( \bigtriangleup(R))\not\subset H_j$ 
for $1 \leq j \leq q$. Then, for any $\epsilon > 0$, the inequality
\begin{eqnarray*}
&~&\sum_{j=1}^q m_f(r, H_j) +  N_{f, ram}(r) \leq 2n T_f(r)\\
&~&+ {(2n+1)^3\over 8} \left((1+\epsilon)(c_f+\epsilon) T_f(r) +\epsilon \log r\right) +  O(\log T_f(r))
\end{eqnarray*}
holds  for all $r \in (0, R)$  outside a set $E$ 
with 
$\int_E \exp((c_f+\epsilon)T_{f}(r))dr <\infty$.
\end{corollary}
It turns out that our treatment of the error term in Nevanlinna's theory permits to extend many of the classical results, using the known strategy. Since the new results seem of interest, for the reader's 
convenience, we repeat the literature in places.
We give in particular a version of Bloch's theorem for maps with values in a complex torus which belong to the space $  {\mathcal E}_0 (\bigtriangleup(1))$.
We also 
  prove a defect relation for the intersection of the image  of a map in   $  {\mathcal E}_0 (\bigtriangleup(1))$ with an ample divisor in an abelian variety extending  results by Siu-Yeung \cite {SY1}.
  \section{Some examples and applications}
In this section, we provide some  examples of holomorphic maps on the unit disc which are in the class we study.

\noindent {\bf Example 1}. Let $N$ be a compact Riemann surface of genus $\ge 2$.
 Then $N$ has a smooth metric form 
$\omega_P$ whose Gauss curvature is $-1$. 
 We take $\phi: \bigtriangleup(1)  \rightarrow N$ 
as the uniformizing map. 
Then 
$$T_{\phi, \omega_P}(r)=\log{1\over 1-r} +O(1).$$
Hence $c_{\phi,  \omega_P}=1$, and thus $\phi\in \mathcal E_1$.
Note that not only we know that $\phi$ is onto but also we   get, from Theorem \ref{g}, that $\delta_{\phi,\omega_P}(a)=0$ for every  $a\in N$.

\noindent {\bf Example 2}. Let $M$ be a compact Kobayashi hyperbolic manifold and let 
$\omega$ be a metric form. Then, by Brody's theorem (see \cite {Lang} or \cite{Ru}), there is a constant $C>0$ such that for any holomorphic map
$f: \bigtriangleup(1) \rightarrow M$,  we have $|f'(0)|_{\omega}\leq C$. Hence 
$|f'(z)|_{\omega}\leq {C\over 1-|z|}$ on $\bigtriangleup(1)$.
Consequently, we have
$T_{f, \omega}(r)\leq C\log{1\over 1-r}$. So the space  $ {\mathcal E}_0$ is empty.
However,  $c_{f, \omega}$ is  not necessarily finite since it requires an estimate on the  lower bound on $T_{f, \omega}(r).$  The 
following two examples give the lower bound on $T_{f, \omega}(r)$ in terms of  $\log{1\over 1-r}$.

\noindent
{\bf Example 3}. Let $(X, \mathcal{L})$ be a compact, 1-dimensional lamination in a compact
Hermitian manifold $(M, \omega)$ (see \cite {DNS}, \cite{FS} and the references therein). Assume that  $(X, \mathcal{L})$ is Brody hyperbolic, which 
means that there is no non-constant image of ${\Bbb C}$ directed by the  lamination  $\mathcal{L}$. 
So for every leave $L$, we have the universal covering map $f:  \bigtriangleup(1) \rightarrow L$.
 It is known (see \cite{FS})  that there are two positive constants  
$C, C'$ (which do not depend on the leave) such that 
$${C\over 1-|\zeta|} \leq |f'(\zeta)|_{\omega} \leq {C'\over 1-|\zeta|}.$$ Therefore 
$$T_{f, \omega}(r)\sim \log{1\over 1-r},$$
so $f\in \mathcal E$.

\noindent
{\bf Example 4}. Let $(M, \omega)$ be a compact Hermitian manifold  and $\mathcal{F}$ be a
 Brody hyperbolic foliation with a finite number of singularities which are linearizable.
According to a result of Dinh-Nguyen-Sibony (See \cite{DNS}), for any extremal positive 
$\partial \bar{\partial}$-closed current $T$ directed by the foliation which gives full mass to hyperbolic leaves, there are two positive constants  
$C, C'$ (which do not depend on the leaves) such that 
$$ C\log{1\over 1-r} \leq T_{\phi, \omega}(r) \leq C'\log{1\over 1-r}$$
for  $T$-almost every leave $L$ (in terms of the measure $T\wedge \omega$). So  $\phi\in \mathcal E$.
 Here
$\phi:  \bigtriangleup \rightarrow L\subset  M$ is the universal covering map of $L$. 

In the case where $\mathcal{F}$  is a foliation in ${\Bbb P}^2({\Bbb C})$, our Theorem \ref{thb} 
implies that, for any line $\Lambda \subset {\Bbb P}^2$, except for countably  many lines,  
there are 
 cluster points of the sequence of the measures
$${1\over T_{\phi}(r)}\sum_{\phi(a)\in \Lambda, |a|<r} \delta_{a} \log^+{r\over |a|}$$
which are  probability measures on the unit circle, where $\delta_a$ is the Dirac measure at $a$. 

   We end this section with the following theorem which characterizes the Kobayashi hyperbolicity of $M$.

\begin{theorem}\label{Ko} Let $M$ be a compact complex manifold. Then the following 
are equivalent.

(a) $M$ is Kobayashi hyperbolic;

(b) For any given positive $(1, 1)$-form $\omega$ on $M$, there are positive constants $c_0$ and $A$ such that for every  holomorphic map
$f: \bigtriangleup(1) \rightarrow M$, $\int_0^1 \exp(c T_{f, \omega}(r))dr \leq A$ for every $c<c_0$;

(c)  The class $\mathcal E_0(\bigtriangleup(1))$ is empty.
\end{theorem}
\begin{proof}  We first prove $(a) \Rightarrow (b)$.  Indeed, since  $M$ is  Kobayashi hyperbolic,
 there is a constant $C>0$ such that for any holomorphic map
$f: \bigtriangleup(1) \rightarrow M$,  we have $|f'(0)|_{\omega}\leq C$. Hence 
$|f'(z)|_{\omega}\leq {C\over 1-|z|}.$ Consequently we have
$T_{f, \omega}(r)\leq C\log{1\over 1-r}$. We take $c_0={1\over 2C}$, then it is easy to see that 
$$\int_0^1 \exp(c T_{f, \omega}(r))dr \leq \int_0^1 {1\over (1-r)^{1/2}}dr  = A$$ for every $c<c_0$. 

The fact that  (b) implies (c) is obvious. So we only need to prove that (c) implies (a).
 It suffices to prove that if $M$ is not Kobayashi hyperbolic then $\mathcal E_0(\bigtriangleup(1))$ is not empty. We first construct a holomorphic map $g: \bigtriangleup(1) \rightarrow {\Bbb C}$, such that for most $a's$, $$\lim_{r \rightarrow 1}{ N_g(r,a)\over \log {1\over 1-r}}=\infty.$$
 Indeed such a holomorphic map  $g_1:  \bigtriangleup(1) \rightarrow \Bbb P^1( {\Bbb C})$ exists (see \cite{Tsuji}). Let $ E$ denote the preimage of the point at infinity in $\Bbb P^1( {\Bbb C})$. We can assume that the point $0$ is not in
 $E$. Let   $ h: \bigtriangleup(1) \rightarrow\bigtriangleup(1)\setminus E$ denote the universal covering map from with $h(0)=0$.  Then the map $g=g_1(h)$ satisfies our condition. 
 
 Since $M$ is not Kobayashi hyperbolic there is a non-constant holomorphic map $f: {\Bbb C} \rightarrow M$.
 The map $F=f(g_1(h))$ satisfies that for most $a's$ 
 \begin{equation}\label{s}\lim_{r \rightarrow 1}{ N_F(r,a)\over \log {1\over 1-r}}=\infty.
 \end{equation} Then a similar growth is valid for $T_F(r)$.  Indeed we have:
 $$ N_F(r,a)=  \int  \log^+{r\over |z|} F^*(\delta_{a}).$$
 Similarly for any positive measure $\mu$ we have
 $$\int N_F(r,a) d\mu(a)=  \int  \log^+{r\over |z|} F^*(\mu).$$
 It suffices to apply this to the form $\omega$ considered as a measure on $F(\bigtriangleup(1))$.
 It follows that if $N(r,a)$ grows fast for most $a's$, the same is true for $T(F,r)$.
 Hence $F\in  \mathcal E_0(\bigtriangleup(1))$ and thus  $\mathcal E_0(\bigtriangleup(1))$ is not empty.
   \end{proof}

\section{Holomorphic mappings  into compact Riemann surfaces}

\begin{lemma}
[Calculus Lemma] Let $0<R\leq \infty$ and  let $\gamma(r)$ be a non-negative function defined 
on $(0, R)$ with 
$\int_0^R \gamma(r) dr=\infty$. Let $h$ be a 
nondecreasing function of class $C^1$ defined on $(0, R)$. 
Assume that $\lim_{r\rightarrow R} h(r)=\infty$ and $h(r_0)\ge c>0$. 
Then, for every $0< \delta<1$,   the inequality
$$ h'(r) \leq h^{1+\delta}(r)   \gamma(r) $$
holds for all $r \in (0, R)$  outside a set $E$ 
with 
$\int_E \gamma(r) dr <\infty$.
\end{lemma}
\begin{proof}
 Let $E \subset (r_0, R)$ be the set of $r$ such that
$h'(r) \ge h^{1+\delta}(r) \gamma(r).$ Then
$$ \int_E \gamma(r)  dr  
 \leq \int_{r_0}^R  {h'(r)\over  h^{1+\delta} (r)} dr
= \int_c^{\infty} {dt\over t^{1+\delta}}  <\infty$$
which proves the lemma.
\end{proof}
\begin{lemma}\label{calculus} Let $0<R\leq \infty$ and  let $\gamma(r)$ be a function defined 
on $(0, R)$ with 
$\int_0^R \gamma(r) dr=\infty$. 
 Let $h$ be a function of class $C^2$  defined on $(0, R)$ such that 
$rh'$ is a  
nondecreasing function. 
Assume that $\lim_{r\rightarrow R} h(r)=\infty$. Then
$${1\over r} {d\over dr}\left(r{dh\over dr}\right)  \leq r^{\delta} \cdot \gamma^{2+\delta} (r) \cdot h^{(1+\delta)^2} (r)$$
holds outside a set $E\subset (0, R)$ 
with 
$\int_E \gamma(r) dr <\infty$.
\end{lemma}

\begin{proof} We apply the Calculus lemma
twice, first to the function $rh'(r)$ and then to the function $h(r)$.
\end{proof}
The typical use of the calculus lemma is as follows.
Let $\Gamma$ be a non-negative function on $\bigtriangleup(R)$ with $0<R\leq \infty$.  Define
$$T_{\Gamma}(r):=\int_0^r {dt\over t} \int_{|z|<t}
\Gamma {{\sqrt -1}\over 2\pi} dz\wedge d{\bar z}$$
 and 
$$\lambda(r):= \int_0^{2\pi} \Gamma(re^{i\theta}) {d\theta\over 2\pi}.$$
Using the polar coordinates,
$$ {{\sqrt -1}\over 2\pi} dz\wedge d{\bar z}
=2rdr \wedge {d\theta\over 2\pi}.$$
Hence
$$r{dT_{\Gamma}\over dr}=2\int_0^{2\pi}\left(\int_0^r 
 \Gamma(te^{i\theta})tdt \right){d\theta\over 2\pi},$$
$$
 {d\over dr}\left(r{dT_{\Gamma}\over dr}\right)=2r
\int_0^{2\pi} \Gamma(re^{i\theta}) {d\theta\over 2\pi}= 2r\lambda(r).$$
Thus, from Lemma \ref{calculus}, we have
\begin{equation}\label{cs}
 \lambda(r)\leq  {1\over 2} r^{\delta}\cdot  \gamma^{2+\delta} (r) \cdot T_{\Gamma}^{(1+\delta)^2} (r)
 \end{equation}
holds for all $r \in (0, R)$  outside a set $E$ 
with 
$\int_E \gamma(r) dr <\infty$.  Throughout the paper, we will use the inequality (\ref{cs}) with a properly chosen $\gamma(r)$.

\begin{theorem}[Green-Jensen formula, see\cite{PS}]
Let $g$ be a function on  $\overline{\bigtriangleup(r)}$
 such that $dd^c [g] $ is of order zero and $g(0)$ is finite.
Then
$$\int_0^r {dt\over t} \int_{|\zeta|< t} dd^c [g] 
= {1\over 2}\left(\int_0^{2\pi} 
g(re^{i\theta}) {d\theta\over 2\pi} -g(0)\right).$$
\end{theorem}

Let $M$ be a compact Riemann surface and let $\omega$ be a positive (1,1) form of class $C^1$ on $M$ such that $\int_M\omega=1$. Consider the equation, in the sense of currents, 
\begin{equation}\label{Poisson}
dd^c u=\omega -\delta_a,
\end{equation}
where $\delta_a$ is the Dirac measure at $a.$
\begin{theorem}
Let $U$ be an open set in a compact Riemann surface $M$ such that $M\backslash U$ consists
of at most a 
finite number of points.

(a) Let  $\omega$  be a positive smooth
(1,1) form of volume 1 on $M$. Let $a\in  M$. Then  
equation (\ref{Poisson}) admits a positive solution $u_a$,  smooth in $M\backslash \{a\}$, with 
a log singularity at the point $a$.

(b)  If $M\backslash U$ is non-empty  and $\omega$ is proportional to  the Poincar\'e form of $M$
so that it is of volume 1, then 
equation  (\ref{Poisson}) admits a positive solution $u_a$,  smooth in $U\backslash \{a\}$, with 
a log singularity at the point $a$.
\end{theorem}
\begin{proof} (a) Since the cohomology class of the right hand side is zero,
 equation  (\ref{Poisson})  always has a solution. The regularity in the complement
of $a$ and the behavior at $a$ imply that $u_a$ is smooth in $M\backslash \{a\}$, with 
a log singularity at the point $a$. By adding a constant if necessary, it  gives the positivity
of $u_a$. This proves the case (a). 

The proof of case (b) is similar. 
Note that the
  Poincar\'e metric at the points in $M\backslash U$ 
behaves like $cdz\wedge d{\bar z}/(|z|^2(\log|z|)^2)$, which has finite volume. Using that the Poincar\'e metric of the pointed disc has curvature $-1$  we can by comparison establish  that 
the solution  $u_a$ goes to $+\infty$ when approaching the points at the boundary. This gives the positivity
of $u_a$.
\end{proof}
Let $a\in U$ and $u_a$ be the solution of the equation  (\ref{Poisson}). We define  the proximity function 
\begin{equation}\label{proximity}
 m_{f, \omega}(r, a)= {1\over 2} \int_0^{2\pi} u_a(f(re^{i\theta})) {d\theta\over 2\pi}
 \end{equation}
 and the counting function
 \begin{equation}\label{counting}
 N_f(r, a)=\int_0^r {n_f(t, a)\over t} dt
 \end{equation}
 where $n(r, a)$ is the number of the elements of $f^{-1}(a)$ inside $|z|<r$, counting multiplicities (for simplicity we assume $0$ is not in $f^{-1}(a)$).

By applying the integral operator
$$\int_0^r {dt\over t} \int_{|\zeta|\leq t} \cdot$$
to the  equation  (\ref{Poisson})
 and using the Green-Jensen's formula, we get
\begin{theorem}[First Main Theorem] $$m_{f, \omega}(r, a) + N_f(r, a) = T_{f, \omega}(r) + O(1).$$
\end{theorem}
The  defect  for $f$ with $c_{f, \omega}<\infty$, is given by,
$$\delta_{f, \omega}(a): =\liminf_{r\rightarrow R} {m_{f, \omega}(r, a)\over T_{f, \omega}(r)} = 
1- \limsup_{r\rightarrow R} {N_f(r, a)\over T_{f, \omega}(r)}, ~~~~~\delta_{f}(a): =\delta_{f, \omega_{FS}}(a).$$

\noindent{\it Proof of Theorem\ref{tha}.}  
Consider
$$\Psi=C \left(\prod_{j=1}^q (u^{-2}_{a_j}\exp(u_{a_j}))\right) \omega
$$
where $C$ is chosen such that $\int_M\Psi=1$. 
Write
$$f^*\Psi=\Gamma { \sqrt{-1}\over 2\pi} d\zeta\wedge d{\bar \zeta}.$$
Then, by the Poincar\'e-Lelong formula,
$$
dd^c [\log \Gamma]
= \sum_{j=1}^q dd^c [u_{a_j}\circ f]+  [f^*\mbox{Ric}(\omega)]
  +
D_{f, \mbox{ram}}- 2\sum_{j=1}^q dd^c [\log u_{a_j}\circ f].$$
Applying the integral operator
$$\int_0^r {dt\over t} \int_{|\zeta|\leq t} \cdot$$
to the above identity and using  the Green-Jensen's formula, we get
\begin{eqnarray*}
{1\over 2}\int_0^{2\pi} \log \Gamma(re^{i\theta}) {d\theta\over 2\pi} + O(1) &=& \sum_{j=1}^q m_f(r, a_j)
+ T_{f, \mbox{Ric}(\omega)}(r)+N_{f, \mbox{ram}}(r)  \\
&& -2\sum_{j=1}^q \int_0^r {dt\over t} \int_{|\zeta|\leq t} dd^c [\log u_{a_j}\circ f].
\end{eqnarray*}
Using the Green-Jensen formula,  the concavity of log and the First Main Theorem, we get
\begin{eqnarray*}
&~&2\int_0^r {dt\over t} \int_{|\zeta|\leq t}  dd^c [\log u_{a_j}\circ f] 
=  \int_0^{2\pi} \log  u_{a_j} (f(re^{i\theta})) {d\theta\over 2\pi} +O(1) \\
&\leq&  \log \int_0^{2\pi}  u_{a_j} (f(re^{i\theta})) {d\theta\over 2\pi}+O(1)= \log m_{f, \omega}(r, a_j)+O(1)\\
&~&\leq \log  T_{f, \omega}(r) +O(1).
\end{eqnarray*}

Using the concavity of $\log$ and   (\ref{cs}) by taking $\gamma(r): =\exp((c_{f, \omega}+\epsilon) T_{f, \omega}(r)) $ and $\delta=2\epsilon$, we have, 
\begin{eqnarray*}
&~&{1\over 2}\int_0^{2\pi}\log \Gamma(re^{i\theta}){d\theta\over 2\pi}
\leq {1\over 2}\log \int_0^{2\pi}  \Gamma (re^{i\theta}){d\theta\over 2\pi} 
+O(1)\\
&\leq& {1\over 2}
 \left((2+2\epsilon)(c_{f, \omega}+\epsilon) T_{f,  \omega}(r) + (1+2\epsilon)^2 \log^+  T_{\Gamma}(r)+2\epsilon \log r \right) +O(1)
\end{eqnarray*}
holds for all $r \in (0, R)$  outside a set $E$ 
with 
$\int_E \exp((c_{f, \omega}+\epsilon) T_{f, \omega}(r))dr <\infty$.
It remains  to estimate
$$T_{\Gamma}(r)
=\int_0^r {dt\over t} \int_{|\zeta|\leq t}
\Gamma {{\sqrt -1}\over 2\pi} d\zeta \wedge d{\bar \zeta}
=\int_0^r {dt\over t} \int_{|\zeta|\leq t}f^* \Psi.$$
 We follow the approach by Ahlfors-Chern.
 The change of variable formula gives, 
$$\int_M n_f(r, a) \Psi(a) = \int_{|\zeta|\leq r} f^*\Psi.$$
So, using the First Main Theorem, 
$$\int_0^r {dt\over t}  \int_{|\zeta|\leq t} f^*\Psi = \int_M N_f(r, a) \Psi(a)
 \leq \int_M T_{f, \omega}(r)  \Psi(a) + O(1) 
=   T_{f, \omega}(r) +O(1).$$  This finishes the proof of Theorem \ref{tha}.

A  similar idea can be carried out to prove Theorem \ref{g}, we have just to use Theorem 3.4 (b).

\noindent{\it Proof of  Theorem \ref{thas}}. 
 Write $f^*\omega = h { \sqrt{-1}\over 2\pi} d\zeta\wedge d{\bar \zeta}.$
Then, by the Poincar\'e-Lelong formula,
$$
dd^c [\log h]= f^*\mbox{Ric}(\omega)
  +
D_{f, \mbox{ram}} =\mbox{Ric}(f^*\omega) + 
D_{f, \mbox{ram}},$$
where $D_{f, \mbox{ram}}$ is the  ramification divisor of $f$.
The curvature assumption implies that  $$dd^c [\log h]\ge D_{f, \mbox{ram}} + \lambda f^*\omega.$$
Applying the integral operator
$$\int_0^r {dt\over t} \int_{|\zeta|\leq t} \cdot$$
to the above identity and using the Green-Jensen's formula, we get
$$
{1\over 2}\int_0^{2\pi} \log h(re^{i\theta}){
d\theta\over 2\pi} +O(1) \ge  \lambda T_{f, \omega}(r)+ N_{f, \mbox{ram}}(r).$$
On the other hand, using the concavity of $\log$ and   (\ref{cs}) by taking $\gamma(r): =\exp((c_{f, \omega}+\epsilon) T_{f, \omega}(r)) $ and  $\delta=2\epsilon$, it follows
\begin{eqnarray*}
\lefteqn{{1\over 2}\int_0^{2\pi}\log h(re^{i\theta}){d\theta\over 2\pi}
\leq {1\over 2}\log \int_0^{2\pi}  h(re^{i\theta}){d\theta\over 2\pi} +O(1)} \\
&\leq& {1\over 2}
 \left((2+2\epsilon) (c_{f, \omega}+\epsilon) T_{f, \omega}(r) + (1+2\epsilon)^2 \log^+  T_{f, \omega}(r)+2\epsilon \log r \right)
\end{eqnarray*}
holds for all $r \in (0, R)$  outside a set $E$ 
with 
$\int_E \exp((c_{f, \omega}+\epsilon) T_{f, \omega}(r))dr <\infty$.
This finishes the proof. 

\section{Holomorphic mappings  into ${\Bbb  P}^n({\Bbb C})$.}
In this section, we prove Theorem \ref{thb}. 
 We follow Ahlfors' method with some simplifications (see  \cite{Ahlfors}, \cite{CG}, \cite{Sh}, \cite{Ru} or  \cite{SW}).
However we treat differently the error term.
The key  is to use  (\ref{cs}) by 
letting $\gamma(r): =\exp((c_f+\epsilon)T_f(r))$ for a given $\epsilon$, 
where $T_f(r):= T_{f, \omega_{FS}}(r)$. In the following we use the notation ``$\leq ~\|$" to denote the inequality holds for all 
$r\in (0, R)$ except for a set $E$ with $\int_E \exp((c_f+\epsilon)T_f(r)) dr<\infty$.   We always  assume that the holomorphic map
 $f: \bigtriangleup(R) \rightarrow {\Bbb P}^n({\Bbb C})$ is linearly non-degenerate (except in  the last section {\bf E}) with $c_f<\infty$ .

\bigskip
\noindent{\bf A. Associated curves and the Pl$\ddot{u}$cker's formula}.   
Let ${\bf f}:\bigtriangleup(R) \rightarrow
{\Bbb C}^{n+1} - \{0\}$  be a reduced representation of $f$.
Consider the holomorphic map  ${\bf F}_k$ defined by
$${\bf F}_k=
{\bf f} \wedge  {\bf f}^\prime
\wedge  \cdots  \wedge  {\bf f}^{(k)}: \bigtriangleup(R) \rightarrow
\bigwedge^{k+1}{\Bbb  C}^{n+1}.$$
Evidently ${\bf F}_{n+1} \equiv 0$. Since 
 $f$ is linearly non-degenerate,  ${\bf F}_{k} \not\equiv 0$
for $0 \leq k \leq n$.
The map $F_k= {\Bbb  P}({\bf F}_k):  \bigtriangleup(R) \rightarrow {\Bbb P}(
\bigwedge^{k+1}{\Bbb  C}^{n+1})= {\Bbb P}^{N_k}({\Bbb C})$,  where
$N_k= {(n+1)!\over (k+1)!(n-k)!} - 1$
and ${\Bbb P}$ is the natural projection,
is called {\bf the $k$-th associated map}. Let $\omega_k=dd^c \log \|Z\|^2$ be
the
Fubini-Study form
on ${\Bbb P}^{N_k}({\Bbb C})$, where $Z=[x_0: \dots: x_{N_k}] \in
{\Bbb P}^{N_k}({\Bbb C})$.
Let
\begin{equation}\label{hk}
\Omega_k = F^*_k\omega_k = {\sqrt{-1}\over 2\pi} h_k dz\wedge d{\bar z},
 ~~0 \leq k \leq n, 
\end{equation}
be the pull-back via the $k$-th associated curve. Observe  that since $F_k$ has no indeterminacy points,  $\Omega_k = F^*_k\omega_k$ is smooth and $h_k$ is non-negative. 

We recall the 
following lemma (see \cite{GH}, \cite{Sh}, \cite{Ru} or  \cite{SW}).

\begin{lemma}\label{h1}  $$h_k=  {
\|{\bf F}_{k-1}\|^2\|{\bf F}_{k+1}\|^2 \over \|{\bf F}_{k}\|^4}.$$
\end{lemma}

We now turn to the Pl$\ddot{u}$cker Formula.
By Lemma \ref{h1} and the Poincar\'e-Lelong formula, we get
\begin{equation}\label{hkexpression}
dd^c \log h_k = \Omega_{k-1} + \Omega_{k+1} - 2\Omega_k+[h_k=0] .
\end{equation}
where $[h_k=0]$ is the zero divisor of $h_k$. We recall a 
few facts on the geometric meaning of this 
divisor (see \cite{GH}, \cite{Sh}). We  consider the point $z_0$ with ${\bf F}_{k}(z_0)=0$. 
Without loss of generality, we assume that $z_0=0$ and $f(z_0)=[1:0:\cdots:0]$
and that the reduced representation ${\bf f}$ of $f$  in a neighborhood of $0$ has the form
$${\bf f}(z)=(1+\cdots , z^{\nu_1}+\cdots, \cdots, z^{\nu_n}+\cdots), $$
with $1\leq \nu_1\leq \cdots \leq\nu_n.$
Then it is easy to get that 
$${\bf F}_k(z)=z^{m_k}(1+\cdots, z^{\nu_{k+1}-\nu_k}+\cdots, \dots),$$
where $m_k= \nu_1+\cdots+\nu_k -{k(k+1)\over 2}$.  On the other hand, if we write  in a neighborhood of $0$, 
$h_k(z)=z^{2\mu_k} b(z)$ with $b(0)>0$, then, it is easy to get $\mu_k=m_{k+1}-2m_k+m_{k-1}$ (see  \cite{GH}).
  
 Define the $k$th characteristic function
$$
T_{F_k}(r) = \int_0^r {dt\over t} \int_{|z| \leq t}  F^*_k\omega_k.
$$
Denote by
$$
N_{d_k}(r) = \int_0^r n_{d_k}(t) {dt\over t}
$$
where  $n_{d_k}(t)$ is the number of zeros of the $h_k$ in $|z| < t$, counting multiplicities. Note that
$N_{d_k}(r, s)$ does not depend on the choice of the reduced representation.
Define
\begin{equation}\label{sk}
S_k(r) = {1\over 2}\int_0^{2\pi} \log h_k(re^{i\theta}) {d\theta\over 2\pi}.
\end{equation}
Then, by applying the integral operator
$$\int_0^r {dt\over t} \int_{|\zeta|\leq t} \cdot$$
to  (\ref{hkexpression}) and using the Green-Jensen's formula, we get the  following lemma.
\begin{lemma}[Pl$\ddot{u}$cker  Formula]\label{h2}
  For  any  integers $k$ with $0 \leq k \leq n$,
\[ N_{d_k}(r) + T_{F_{k-1}}(r) - 2 T_{F_{k}}(r) + T_{F_{k+1}}(r)
= S_k(r) +O(1)\]
where $T_{F_{-1}}(r) \equiv 0$ and  $T_{F_0}(r) = T_f(r)$.
\end{lemma}
The Pl$\ddot{u}$cker formula implies the following lemma which gives the estimates of $ T_{F_k}(r)$ 
in terms of $T_f(r)$. We use our estimate of the error term. 
\begin{lemma}\label{h3}
For $0 \leq k \leq n-1$ and every $\delta>0$, 
$$ T_{F_k}(r) \leq  (n+2)^3(1+ (2+\delta) c_f) T_f(r) + n(n+1)^2 \delta\log r +O(1)~\|.$$
\end{lemma}
\begin{proof} Write $T(r) = \sum_{k=0}^{n-1} T_{F_k}(r).$
Observe that
 $${1\over r}{d\over dr}\left(r{dT_{F_k}(r)\over dr}\right)
= 2 \int_0^{2\pi}  h_k (re^{i\theta}) {d\theta\over 2\pi}.$$
Applying the Calculus Lemma (see (\ref{cs})) with $\gamma(r) =\exp((c_f+\delta) T_f(r))$, we get
$$\int_0^{2\pi}  h_k(re^{i\theta}) {d\theta\over 2\pi} \leq r^{2\delta}  e^{c_f(4+2\delta) T_f(r)} 
T^{(1+2\delta)^2}_{F_k}(r) ~\|.$$
This implies 
\begin{eqnarray}\label{S}
S_k(r) &=& {1\over 2}\int_0^{2\pi} \log h_k(re^{i\theta}) {d\theta\over 2\pi}\nonumber \\
&\leq& {1\over 2}\log \int_0^{2\pi} h_k(re^{i\theta}) {d\theta\over 2\pi}+O(1\nonumber )\\
&\leq&(2+\delta)c_f T_f(r) +{1\over 2} (1+2 \delta)^2 \log T(r)  + \delta \log r~\|.
\end{eqnarray}
From Lemma \ref{h2}, we claim that, for $0 \leq q \leq p$,
$$T_{F_p}(r) + (p-q)T_{F_{q-1}}(r)\leq (p-q+1)T_{F_q}(r) + \sum_{j=q}^{p-1} (p-j) S_j(r) +O(1) .$$
In fact, the claim is true for $p=q$. Assume that the claim is true for $q, q+1, \dots, p$.
If $p=n$, the proof is done. If $p < n$, we proceed, by using Lemma \ref{h2}, 
\begin{eqnarray*}
&~&T_{F_{q-1}}(r) -  T_{F_{q}}(r) + T_{F_{p+1}}(r) - T_{F_{p}}(r)\\
&=&
\sum_{j=q}^p \left(T_{F_{j-1}}(r) -  2T_{F_{j}}(r) + T_{F_{j+1}}(r)\right)
=\sum_{j=q}^{p}  S_j(r) -\sum_{j=q}^p  N_{d_j}(r) +O(1)\\
&\leq&  \sum_{j=q}^{p}  S_j(r) +O(1).
\end{eqnarray*}
So
$$T_{F_{p+1}}(r) + T_{F_{q-1}}(r) \leq   T_{F_{p}}(r) +  T_{F_{q}}(r) + \sum_{j=q}^p  S_j(r) +O(1).$$
Thus
\begin{eqnarray*}
&~&T_{F_{p+1}}(r) + (p+1-q)T_{F_{q-1}}(r) =  T_{F_{p+1}}(r) + T_{F_{q-1}}(r) + (p-q) T_{F_{q-1}}(r)  \\
&\leq & T_{F_{p}}(r) +  T_{F_{q}}(r)+(p-q) T_{F_{q-1}}(r)  + \sum_{j=q}^p  S_j(r) +O(1).
\end{eqnarray*}
On the other hand, from Lemma \ref{h2} again, we have
\begin{eqnarray*}
&~&T_{F_p}(r) - (p-q+1)  T_{F_q}(r) +(p-q) T_{F_{q-1}}(r) \\
&=&\sum_{j=q}^p (p-j) \left(T_{F_{j-1}}(r) -  2T_{F_{j}}(r) +T_{F_{j+1}}(r)\right)
\leq  \sum_{j=q}^{p} (p-j) S_j(r) + O(1).
\end{eqnarray*}
Hence
$$ T_{F_{p}}(r) +  T_{F_{q}}(r)+  (p-q) T_{F_{q-1}}(r) \leq (p-q+2)  T_{F_q}(r) + \sum_{j=q}^{p} (p-j) S_j(r) + O(1).$$
Therefore
$$T_{F_{p+1}}(r) + (p+1-q)T_{F_{q-1}}(r)
\leq (p-q+2)  T_{F_q}(r) + \sum_{j=q}^{p} (p+1-j) S_j(r) +  O(1).$$
This proves our claim. Now take $q=0$ and $p=k$ and notice that $T_{F_{-1}}(r)\equiv 0$, then
$$T_{F_k}(r) \leq (k+1)T_f(r) +  \sum_{j=0}^{k-1} (k-j) S_j(r) +O(1).$$
This, together with (\ref{S}) gives,  for $0 \leq k \leq n$,
\begin{eqnarray*}
&~&T_{F_k}(r) \leq  (k+1) T_f(r)\\
&~& + {1\over 2}k(k+1) \left((2+\delta)c_f T_f(r)  +
(1+2\delta)^2  \log T(r) +\delta\log r+ O(1)\right)~\|.
\end{eqnarray*}
Therefore,
\begin{eqnarray*}
&~& T(r)  \leq  (n+1)^2 T_f(r) \\
&~& +  {1\over 2}n(n+1)^2 \left( (2+\delta) c_f T_f(r)  +
{1\over 2}(1+2\delta)^2 \log T(r) +\delta\log r+ O(1)\right)~\|.
\end{eqnarray*}
Because
${1\over 2}n(n+1)^2(1+2\delta)^2 \log T(r) \leq {1\over 2} T_f(r)$ where $r$ is close enough to $R$, 
 we have
$$T(r)  \leq  (n+2)^3(1+ (2+\delta) c_f) T_f(r) + n(n+1)^2 \delta\log r +O(1)~\|.$$
\end{proof}

\noindent{\bf B. The projective distance}.~For integers $1 \leq q \leq p \leq n+1$,
the {\bf interior product} $\xi \lfloor \alpha \in \bigwedge ^{p-q} {\Bbb C}^{n+1}$
of vectors $\xi \in
\bigwedge ^{p+1} {\Bbb C}^{n+1}$ and $\alpha \in  \bigwedge ^{q+1} ({\Bbb C}^{n+1})^*$
is defined by
$$\beta(\xi \lfloor \alpha) = (\alpha\wedge \beta)(\xi)$$
for any $\beta \in  \bigwedge ^{p-q}( {\Bbb C}^{n+1})^*$.
Let
$$H = \{[x_0: \cdots:x_n] ~|~a_0 x_0 + \cdots + a_n x_n = 0\}$$
 be a hyperplane in ${\Bbb P}^n({\Bbb C})$ with unit normal vector
${\bf a} = (a_0, \cdots,a_n)$. In the rest of this section, we 
 regard ${\bf a}$ as a vector
in $({\Bbb C}^{n+1})^*$ which is defined by ${\bf a}({\bf x})
= a_0x_0 + \cdots + a_nx_n$ for each ${\bf x} =  (x_0, \cdots, x_n)
\in {\Bbb C}^{n+1}$, where $({\Bbb C}^{n+1})^*$ is the dual space of ${\Bbb C}^{n+1}$.
Let $x \in {\Bbb P}(\bigwedge^{k+1} {\Bbb C}^{n+1})$,
the {\bf projective distance} is defined by
\begin{equation}
 \|x; H\| = {\|\xi \lfloor {\bf a} \| \over \|\xi\| \|{\bf a} \|}
\end{equation}
where   $\xi \in \bigwedge^{k+1}{\Bbb C}^{n+1}$ with ${\Bbb P}(\xi) = x$.
Define
\begin{equation}
m_{F_k}(r, H) =
\int_0^{2\pi} \log {1\over \|F_k(re^{i\theta}); H\|} {d\theta\over 2\pi}.
\end{equation}
 We have 
the following weak form of the First Main Theorem for $F_k$.
\begin{theorem}[Weak First Main Theorem]\label{first}
 $$m_{F_k}(r, H) \leq  T_{F_k}(r) + O(1).$$
\end{theorem}
\begin{proof} Let ${\bf f}_k:  \bigtriangleup(R)\rightarrow \bigwedge^{k+1} {\Bbb C}^{n+1}$ be a reduced representation of $F_k$,
and we consider the holomorphic map
$$F_k\lfloor {\bf a}: \bigtriangleup(R) \rightarrow {\Bbb P}(\bigwedge^{k} {\Bbb C}^{n+1})$$ which is given by 
$F_k\lfloor {\bf a}:= {\Bbb P}(G)$ where 
$G= {\bf f}_k \lfloor {\bf a}$. Note that  $G$ is a representation of the holomorphic map $F_k\lfloor {\bf a}$, but is not reduced.
We denote by $\nu_G$ the  divisor of $G$ on $\bigtriangleup(R)$, and $N_{G}(r, 0)$ the counting function 
associated to $\nu_G$ (which is independent of the choices of  the reduced representation of $F_k$).
We have
$$ (F_k\lfloor {\bf a})^*\omega_k+ \nu_G = dd^c\log \|G\|^2.$$
Applying the integral operator
$$\int_0^r {dt\over t} \int_{|\zeta|\leq t}$$
to the above identity and  using the Green-Jensen's formula, we get
\begin{eqnarray*}
T_{F_k\lfloor {\bf a}}(r) +  N_G(r, 0)&=&\int_0^{2\pi} \log \|G(re^{i\theta})\|{d\theta\over 2\pi}  +O(1) \\
&=& \int_0^{2\pi} \log \|{\bf f}_k \lfloor {\bf a} \| (re^{i\theta}){d\theta\over 2\pi}+O(1).
\end{eqnarray*}
On the other hand, from the definition (notice that ${\bf f}_k$ is a  reduced representation of $F_k$),
$$T_{F_k}(r) = \int_0^{2\pi} \log \|{\bf f}_k\| (re^{i\theta}){d\theta\over 2\pi}  +O(1).$$
Hence, from the definition of $m_{F_k}(r, H)$,
\begin{eqnarray*}
&~&T_{F_k\lfloor {\bf a}}(r) +  N_G(r, 0) + m_{F_k}(r, H) \\
&=& \int_0^{2\pi} \log \|{\bf f}_k \lfloor {\bf a} \| (re^{i\theta}){d\theta\over 2\pi}+O(1) + \int_0^{2\pi} \log {\|{\bf f}_k\| \|{\bf a}\|\over  \|{\bf f}_k \lfloor {\bf a} \|}(re^{i\theta}){d\theta\over 2\pi}\\
 &=& \int_0^{2\pi} \log \|{\bf f}_k\| (re^{i\theta}){d\theta\over 2\pi}  +O(1) 
 =   T_{F_k}(r)+O(1).
\end{eqnarray*}
\end{proof}
 
 We shall need the following product to sum estimate. It is an extension
of the estimate of the geometric mean by the arithmetic mean. \begin{lemma}
[See  Theorem 3.5.7 in \cite{Ru}]\label{product3}
 Let  $H_1, \dots, H_q$ (or ${\bf a}_1, \dots, {\bf a}_q$)
 be hyperplanes in ${\Bbb P}^n({\Bbb C})$ in
general position.
Let  $k \in {\Bbb Z}[0,n-1]$ with $n-k \leq q$. Then there exists a
constant $c_k > 0$ such that for every $0< \lambda < 1$ and
  $x \in {\Bbb P}(\bigwedge ^k {\Bbb C}^{n+1})$ with $x \not\subset H_j, 1 \leq j \leq q$ and
$y \in {\Bbb P}(\bigwedge ^{k+1} {\Bbb C}^{n+1})$ we have
$$ \prod_{j=1}^q {\|y; H_j\|^2\over \|x; H_j\|^{2-2\lambda}}
\leq c_k \left( \sum_{j=1}^q {\|y; H_j\|^2\over \|x;  H_j\|^{2-2\lambda}}\right)^{n-k}.$$
\end{lemma}
\noindent{\bf C.  The Ahlfors' estimate}.
Let $\phi_k(H) = \|F_k; H\|^2$.  Define
\begin{equation}\label{h} h_k(H) = {\phi_{k-1}(H)\phi_{k+1}(H)\over \phi_k^2(H)} \Omega_k.
\end{equation}
The function $\phi_k(H)$ is defined out of the stationary points, however the analysis near those points shows that $\phi_k(H)$ can be extended smoothly at those points \cite{Sh}. The key of this Ahlfors' approach  is the following so-called Ahlfors' estimate. We include a proof here.

\begin{theorem}[Ahlfors' estimate (\cite{Ru} or  \cite{SW}]\label{alf} 
 Let $H$ be a
hyperplane in ${\Bbb P}^n({\Bbb C})$.
 Then for any $0 < \lambda < 1$,
 we have
\[ \int_0^r \int_{|z| < t}  {\phi _{k+1}(H) \over
\phi_k(H)^{1 - \lambda}} \Omega_k{dt\over t}
\leq {1 \over \lambda^2}(8T_{F_k}(r) + O(1)).\]
\end{theorem}
To prove Ahlfors' estimate, the following lemma plays a crucial role  (see \cite{Sh}, \cite{Ru} or  \cite{SW}). The proof of the lemma is based on 
a standard but lengthy computation. For the details of the proof, see Lemma A3.5.10 in \cite{Ru}.
\begin{lemma}[Lemma A3.5.10 in \cite{Ru}]\label{alfl} Let $H$ be a  hyperplane  ${\Bbb P}^n({\Bbb C})$ and 
 $\lambda$ be a constant with $0 < \lambda < 1$. Then,  for
$0 \leq k \leq n$, the following inequality holds on 
$\bigtriangleup(R) - \{z~|~ \phi_k(H)(z) = 0\}$
$${\lambda^2\over 4}{\phi_{k+1}(H)\over \phi_k^{1-\lambda}(H)}\Omega_k
- {\lambda(1+\lambda)} \Omega_k \leq dd^c \log(1+\phi_k(H)^{\lambda}).$$
\end{lemma}

We now prove Theorem \ref{alf} (Ahlfors' Estimate).
\begin{proof}
 By Lemma \ref{alfl},
$$dd^c \log(1+\phi_k(H)^{\lambda}) \ge{\lambda^2\over 4}{\phi_{k+1}(H)\over \phi_k^{1-\lambda}(H)}\Omega_k
- {\lambda(1+\lambda)} \Omega_k.$$
Thus
\begin{equation}\label{phi}{\lambda^2\over 4}{\phi_{k+1}(H)\over \phi_k^{1-\lambda}(H)}\Omega_k
\leq dd^c \log(1+\phi_k(H)^{\lambda})
+ \lambda(1+\lambda) \Omega_k.
\end{equation}
By the Green-Jensen's formula,
\begin{eqnarray*}
\lefteqn{\int_0^r {dt\over t} \int_{|z| \leq t} dd^c \log(1+\phi_k(H)^{\lambda})}\\
&=& {1\over 2} \int_0^{2\pi} \log (1+  \phi_k(H)^{\lambda}){d\theta\over
2\pi} +O(1)
\end{eqnarray*}
This, together with (\ref{phi})
implies that
\begin{eqnarray*}
&&{\lambda^2\over 4} \int_0^r {dt\over t} \int_{|z| \leq t}
{\phi_{k+1}(H)\over \phi_k^{1-\lambda}(H)}\Omega_k \\
&&\leq
\int_0^r {dt\over t} \int_{|z| \leq t} dd^c \log(1+\phi_k(H)^{\lambda})
+  \lambda(1+\lambda) T_{F_k}(r)\\
&~&=  {1\over 2} \int_0^{2\pi}  \log (1+  \phi_k(H)^{\lambda}){d\theta\over
2\pi}  +  \lambda(1+\lambda) T_{F_k}(r) +O(1) \\
&& \leq  \lambda(1+\lambda) T_{F_k}(r) + {1\over 2} \log 2+O(1)
\leq 2T_{F_k}(r) +O(1),
\end{eqnarray*}
using  $0 \leq  \phi_k(H) \leq 1$.
\end{proof}

\noindent{\bf D.  A general theorem}.  We prove the following general version of H. Cartan's theorem.
\begin{theorem}[A General Form of the SMT]\label{gen} 
$f:   \bigtriangleup(R) \rightarrow {\Bbb  P}^n({\Bbb C})$ be a linearly non-degenerate
 holomorphic curve (i.e. its image is not
 contained in any proper subspace of  ${\Bbb  P}^n({\Bbb C})$) with $c_f<\infty$, where $c_f=c_{f, \omega_{FS}}$ and $0<R\leq \infty$. 
   Let $H_1,...,H_q$  $($or linear forms ${\bf a}_1, \dots, {\bf a}_q$$)$ be arbitrary
 hyperplanes in ${\Bbb P}^n({\Bbb C})$. Then, for any $\epsilon > 0$, the inequality
\begin{eqnarray*}
&~&\int_0^{2\pi}\max_K\sum_{j \in K}
\log{1\over \|f(re^{i\theta}); H_j\|} {d\theta\over 2\pi}  + N_{W}(r, 0)\\
&\leq& (n+1)T_f(r) + {n(n+1)\over 2} (1+\epsilon)(c_f+\epsilon) T_f(r)\\
&~& + O(\log T_f(r)) +{n(n+1)\over 2}\epsilon \log r~\|,
\end{eqnarray*}
where 
 the max is taken over all subsets $K$ of $\{1,\dots,q\}$ such that the 
linear forms ${\bf a}_j, j \in K$, 
are linearly independent.
\end{theorem}
\begin{proof}
Without  loss of generality, we may assume $q \ge n+1$ and 
that $\#K = n+1$.   Let $T$ be the set of all 
the injective maps $\mu: \{0,1,\dots,n\} \rightarrow \{1,\dots,q\}$ such that 
${\bf a}_{\mu(0)},\dots, {\bf a}_{\mu(n)}$ are linearly independent. 
Take 
\begin{equation}\label{gamma}
\lambda:=\Lambda(r)=\min_k \left\{{1\over T_{F_k}(r)}\right\}.
\end{equation}
For any $\mu \in T$, by Lemma \ref{product3} with $\lambda = \Lambda (r)$ and 
notice that $\phi_k(H) =\|F_k, H\|^2$, it gives, for $0\leq k\leq n-1$, 
\[ \prod_{j=0}^n
 {\phi_{k+1}(H_{\mu(j)})\over \phi_k (H_{\mu(j)})^{1-\Lambda(r)} }
\leq c_k \left(\sum_{j=0}^n {\phi_{k+1}(H_{\mu(j)})\over \phi_k (H_{\mu(j)})^{1-\Lambda(r)} }\right)^{n-k}\]
for some constant $c_k>0$. 
Since $\phi_n(H_{\mu(j)})$ is a constant for any $0 \leq j \leq n$ and $F_0=f$, the above inequality implies that 
\[ \prod_{j=0}^n {1 \over \|f; H_{\mu(j)}\|^{2}}
\leq c \prod_{k=0}^{n-1} \left(\sum_{j=0}^n {\phi_{k+1}(H_{\mu(j)})\over \phi_k (H_{\mu(j)})^{1-\Lambda(r)} }\right)^{n-k}
\cdot \prod_{k=0}^{n-1}\prod_{j=0}^n {1 \over \phi_k(H_{\mu(j)})^{\Lambda(r)}}\]
for some constant $c>0$.
Therefore
\begin{eqnarray}  
\lefteqn{\int_0^{2\pi}\max_K\sum_{j \in K}
\log{1\over \|f(re^{i\theta}); H_j\|^{2}}{d\theta\over 2\pi}\nonumber= \int_0^{2\pi}\max_{\mu \in T}\log
\prod_{j=0}^n{1\over \|f(re^{i\theta}); H_{\mu(j)}\|^2}{d\theta\over 2\pi}} \\
&\leq&  \sum_{k=0}^{n-1}\int_0^{2\pi} \max_{\mu \in T}\log \left(\sum_{j=0}^n {\phi_{k+1}(H_{\mu(j)})\over \phi_k (H_{\mu(j)})^{1-\Lambda(r)} } (re^{i\theta})\right)^{n-k}  {d\theta\over 2\pi}\nonumber  \\
&&+ \sum_{k=0}^{n-1}\sum_{j=0}^n\int_0^{2\pi}\max_{\mu \in T}
\log {1 \over \phi_k(H_{\mu(j)})^{\Lambda(r)}(re^{i\theta}) }{d\theta\over 2\pi}  + O(1)\nonumber\\
&=& \sum_{k=0}^{n-1} (n-k) \int_0^{2\pi} \max_{\mu \in T}\log \left(\sum_{j=0}^n {\phi_{k+1}(H_{\mu(j)})\over \phi_k (H_{\mu(j)})^{1-\Lambda(r)}} (re^{i\theta})\cdot 
 h_k(re^{i\theta}) \right) {d\theta\over 2\pi} \nonumber \\
&~&-2\sum_{k=0}^{n-1}(n-k) S_k(r) 
+ \sum_{k=0}^{n-1}\sum_{j=0}^n\int_0^{2\pi}\max_{\mu \in T}
\log {1 \over \phi _k(H_{\mu(j)})^{\Lambda(r)} (re^{i\theta})}
{d\theta\over 2\pi}  + O(1), \nonumber
\end{eqnarray}
where $h_k$ is defined in (\ref{hk}). 
By Lemma \ref{h2}, noticing that $N_W(r, 0)=N_{d_n}(r)$, we have
\begin{eqnarray*}
&~&\sum_{k=0}^{n-1} (n-k)S_k(r) = \sum_{k=0}^{n-1} (n-k) N_{d_k}(r) \\
&~&+ \sum_{k=0}^{n-1} (n-k) (T_{F_{k-1}}(r)- 
2 T_{F _{k}}(r) + T_{F _{k+1}}(r)) +O(1)\\
&=& N_{d_n}(r) -(n+1)T_f(r) + O(1) = N_W(r, 0)-(n+1)T_f(r)+O(1).
\end{eqnarray*}
Also, by  Theorem \ref{first} (the weak First Main Theorem) and (\ref{gamma}), 
\begin{eqnarray*}
\lefteqn{\sum_{k=0}^{n-1}\sum_{j=0}^n\int_0^{2\pi}\max_{\mu \in T}
\log {1 \over \phi _k(H_{\mu(j)})^{\Lambda(r)} (re^{i\theta})}
{d\theta\over 2\pi}   } \nonumber\\
&=& \sum_{\mu \in T}\sum_{k=0}^{n-1}\sum_{j=0}^n 2\Lambda(r) 
m_{F _k}(r, H_{\mu(j)}) + O(1) \nonumber\\
&\leq& \sum_{k=0}^{n-1}\sum_{j=0}^n  2q!\Lambda(r) 
T_{F _k}(r) + O(1)  \nonumber
\leq O(1).
\end{eqnarray*}
So
\begin{equation}\label{final}
\int_0^{2\pi}\max_K\sum_{j \in K}
\log{1\over \|f(re^{i\theta}); H_j\|}{d\theta\over 2\pi} \leq (n+1)T_f(r) -N_W(r, 0) +  G(r)+O(1),
\end{equation}
where $$G(r)={1\over 2} \sum_{k=0}^{n-1}(n-k) \int_0^{2\pi} \max_{\mu \in T}\log \left(\sum_{j=0}^n {\phi_{k+1}(H_{\mu(j)})\over 
\phi_k(H_{\mu(j)})^{1-\Lambda(r)}}(re^{i\theta})\cdot h_k(re^{i\theta})\right){d\theta\over 2\pi}.$$
We now estimate $G(r)$.
Let 
$$
\hat T(r): = \int_0^r \left(\int_{|z| < t} {\phi _{k+1}(H) \over \phi _{k}(H)^{1 - \Lambda(r)}} h_k {\sqrt{-1}\over 2\pi} dz\wedge d{\bar z}\right)
 {dt\over t}.$$
 Then, from Theorem \ref{alf},  (\ref{gamma}) and Lemma \ref{h3}, we get 
 \begin{equation}\label{hat}\hat T(r)\leq O(T^3_{F_{k}}(r)) =O(T^3_f(r)).
 \end{equation}
Then, by (\ref{cs}) with $\gamma(r)=e^{(c_f+\epsilon)T_f(r)}$,   for every hyperplane $H$, 
$$\int_0^{2\pi} {\phi_{k+1}(H)(re^{i\theta})\over \phi_k(H)^{1- \Lambda(r)} (re^{i\theta})} 
h_k(re^{i\theta}) {d\theta\over 2\pi} \leq r^{2\epsilon}  e^{(c_f+\epsilon)(2+2\epsilon)T_f(r)} \cdot \hat T^{(1+2\epsilon)^2}(r)~\|.$$
 This, together with the  concavity of $\log$ and (\ref{hat}), gives
 \begin{eqnarray*}
G(r)&=&{1\over 2} \sum_{k=0}^{n-1}(n-k) \int_0^{2\pi} \log \max_{\mu \in T}\sum_{j=0}^n {\phi_{k+1}(H_{\mu(j)})\over 
\phi_k(H_{\mu(j)})^{1-\Lambda(r)}} h_k(re^{i\theta}) {d\theta\over 2\pi}\\
&\leq&  \sum_{k=0}^{n-1} {n-k\over 2}\log  \int_0^{2\pi}  \sum_{j=1}^q {\phi_{k+1}(H_j)(re^{i\theta})\over \phi_k(H_j)^{1- \Lambda(r)} (re^{i\theta})} 
h_k(re^{i\theta}) {d\theta\over 2\pi}+O(1)\\
&\leq&  ((c_f+\epsilon)(2+2\epsilon)T_f(r)+2\epsilon\log r)\sum_{k=0}^{n-1} {n-k\over 2}  + 
O(\log T_f(r)) \|\\
& =&  {n(n+1)\over 2} \left( (1+\epsilon)(c_f+\epsilon)T_f(r)  + \epsilon \log r\right) + O(\log T_f(r))~\|.
\end{eqnarray*}
Combining this with (\ref{final}) proves Theorem \ref{gen}.
\end{proof}

\noindent{\bf E. The proof of Theorem \ref{thb}}.  We first consider the case when $k=n$, i.e. $f$ is linearly non-degenerate. We need the following lemma.
\begin{lemma}[see Lemma  A3.1.6 in \cite{Ru}]
 Let $H_1,\dots,
H_q$ be hyperplanes in ${\Bbb P}^n({\Bbb C})$ in general position.
Denote by T
the set of all injective maps $\mu: \{0,1,\dots,n\} \rightarrow 
\{1,\dots,q\}$. 
Then 
$$ \sum_{j=1}^q  m_f(r, H_j) \leq 
\int_0^{2\pi}\max_{\mu \in T}\sum_{i=0}^n \log{1\over \|f(re^{i\theta}); H_{\mu(i)}\|}
 {d\theta\over 2\pi} + O(1). $$
 \end{lemma}
Theorem \ref{gen}, together with the above Lemma, proves Theorem \ref{thb} in this case. 

We now deal with  the case when $f$ is degenerate. By the assumption, we can assume that  
 $f(\bigtriangleup(R))\subset {\Bbb P}^k(\Bbb C)$ with $0\leq k <n$ and $f$ becomes linearly non-degenerate. 
We also assume that $q \ge 2n-k+1$.
Denote by ${\hat H}_j = H_j\cap {\Bbb P}^k({\Bbb C})$. Then ${\hat H}_j$ are hyperplanes 
in ${\Bbb  P}^k({\Bbb C})$ located in n-subgeneral position. Here hyperplanes $H_1, \dots, H_q $ (or $\bf a_1, 
\dots, \bf a_q$)  in ${\Bbb P}^k({\Bbb C})$ are said to be  in {\it $n$-subgeneral position} if, for every 
$1 \leq i_0 < \cdots < i_n \leq q$, 
 the linear span of ${\bf a}_{i_0}, \dots, {\bf a}_{i_n}$
is ${{\Bbb C}^{k+1}}^*$.
We recall the following result due to Nochka.
\begin{lemma}[See Theorem A3.4.3 in \cite{Ru}]\label{Nochka}
Let $H_1, \dots, H_q$ (or ${\bf a}_1, \dots, {\bf a}_q$) be hyperplanes 
in ${\Bbb P}^k({\Bbb C})$ in $n-$subgeneral positions with  $2n - k +1 \leq q$.
Then there exists a function $\omega:
\{1,\dots,q\} \rightarrow (0,1]$ called a Nochka weight and a real
number $\theta \ge 1$ called Nochka constant satisfying the following
properties:

(i) If $j \in \{1,\dots,q\}$, then $0 \leq \omega(j) \theta \leq 1$.

(ii) $q - 2n + k -1 = \theta(\sum_{j=1}^q \omega(j) - k -1).$

(iii) If $\emptyset \not= B \subset \{1,\dots,q\}$ with $\# B \leq n+1$, then
$\sum_{j \in B}\omega(j) \leq \dim L(B)$, where  $L(B)$ is the linear
space generated by $\{{\bf a}_j | j \in B \}$, 

(iv) $1 \leq (n+1)/(k+1) \leq \theta \leq (2n - k +1)/(k+1)$.

(v) Given real numbers $E_1,\dots,E_q$ with $E_j \ge 1$ for $1 \leq j \leq q$,
 and given 
 any $Y \subset \{1,\dots,q\}$ with $0 < \#Y \leq n+1,$
 there exists a subset M of Y with $\#M = \dim L(Y)$ such that $\{{\bf a}_j
\}_{j \in M}$ is a basis for $L(Y)$ where  $L(Y)$ is the linear
space generated by $\{{\bf a}_j | j \in Y \}$, 
 and 
\[ \prod_{j \in Y}
E_j^{\omega(j)}
\leq \prod_{j \in M} E_j.\]
\end{lemma}
We now continue our proof. Since $H_1, \dots, H_q$ (or ${\bf a}_1, \dots, {\bf a}_q$) are 
hyperplanes in $n$-subgeneral position, 
for each $z \in   \bigtriangleup(R)$, there are  (see the proof of Lemma B3.4.4 in \cite{Ru} for 
detail)
indices $i(z, 0), \dots, i(z, n) \in \{1, \dots, q\}$ such that  
\begin{equation}\label{u}
\prod_{j=1}^q {1\over \|f(z); H_j\|^{\omega(j)}} \leq 
C\prod_{l=0}^n {1\over \|f(z); {\hat H}_{i(z,l)}\|^{\omega(i(z,l))}}
\end{equation}
where $\omega(j)$  is the Nochka weight corresponding to ${\hat H}_j$
and $C>0$ is a constant.
Applying Lemma \ref{Nochka} with
$$E_l
= {1\over  \|f(z); {\hat H}_{i(z,l)}\|}, ~~~0 \leq l \leq n,$$ there is a subset $M$ of $Y=\{i(z,0),\dots,i(z,n)\}$ with $\#M=
k+1$ such that $\{{\hat H}_{i(z,j)} | i(z,j) \in M \}
$ is linearly independent, and 
$$\prod_{l=0}^n {1\over  \|f(z); {\hat H}_{i(z,l)}\|^{\omega(i(z, l))}} \leq \prod_{i(z,j) \in M} {1\over \|f(z); {\hat H}_{i(z,l)}\|}.$$
Thus, together with (\ref{u}), 
$$
\prod_{j=1}^q {1\over \|f(z); H_j\|^{\omega(j)}} \leq 
C \max_{\gamma \in \Gamma}\prod_{l=0}^{k} {1\over \|f(z); {\hat H}_{\gamma(l)}\|}
$$
where $\Gamma$ is the set of all maps $\gamma: \{0,\dots, k\} \rightarrow
\{1,\dots,q\}$ such that
${\hat H}_{\gamma(0)}, \dots, {\hat H}_{\gamma(k)}$ are linearly independent.
Hence, by applying the integration, we get,  together with Theorem \ref{gen}, 
\begin{eqnarray*}
\lefteqn{\sum_{j=1}^q \omega(j)m_f(H_j, r) \leq \int_0^{2\pi}
 \max_{\gamma \in \Gamma}\sum_{l=0}^{k}\log {1\over \|f(re^{i\theta}); {\hat H}_{\gamma(l)}\|} 
  {d\theta \over 2\pi} + O(1)}\\
&\leq& (k+1)T_f(r) - N_{f, ram}(r)+ {k(k+1)\over 2} (1+\epsilon)(c_f+\epsilon) T_f(r) \\
&~&+ O(\log T_f(r)) +{k(k+1)\over 2}\epsilon \log r ~\|.
\end{eqnarray*}
By Lemma \ref{Nochka}, and recalling that $m_f(r, H_j) \leq T_f(r) + O(1)$,  it gives
\begin{eqnarray*}
\lefteqn{\sum_{j=1}^q m_f(r, H_j)
=  \sum_{j=1}^q (1-\theta\omega(j))m_f(r, H_j) +  
\sum_{j=1}^q \theta\omega(j)m_f(r, H_j)}\\ 
&\leq& \sum_{j=1}^q (1-\theta\omega(j))m_f(r, H_j) +  \theta(k+1)T_f(r)
-\theta N_{f, ram}(r) \\
&~& + ~\theta{k(k+1)\over 2}\left((1+\epsilon)(c_f+\epsilon) T_f(r) +\epsilon \log r\right) +  O(\log T_f(r)) \\
&\leq& \sum_{j=1}^q (1-\theta\omega(j))T_f(r) + \theta(k+1)T_f(r)
-\left({n+1\over k+1}\right)N_{f, ram}(r) \\
&~& + ~{(2n -k+1)k\over 2}\left((1+\epsilon)(c_f+\epsilon) T_f(r) +\epsilon \log r\right) +  O(\log T_f(r))\\
&=& \left\{q-\theta\left(\sum_{1 \leq j \leq q}\omega(j) - 
k - 1\right)\right\} T_f(r) -\left({n+1\over 
k+1}\right)N_{f, ram}(r) \\
&~& +~{(2n -k+1)k\over 2}\left((1+\epsilon)(c_f+\epsilon) T_f(r) +\epsilon \log r\right) +  O(\log T_f(r))\\
&=& (2n-k+1)T_f(r)- \left({n+1\over k+1}\right)N_{f, ram}(r)\\
&~& + ~{(2n -k+1)k\over 2}\left((1+\epsilon)(c_f+\epsilon) T_f(r) +\epsilon \log r\right) +  O(\log T_f(r)),
\end{eqnarray*}
where the inequality holds  for all $r \in (0, R)$  outside a set $E$ 
with 
$\int_E \exp((c_f+\epsilon)T_{f}(r))dr <\infty$.
This proves Theorem \ref{thb}.

\section{The  Logarithmic Derivative Lemma and the fundamental vanishing theorem}
We begin with the following Logarithmic Derivative Lemma for meromorphic functions. 
\begin{theorem}[Logarithmic Derivative Lemma]\label{l}  Let $0<R\leq \infty$ and  let $\gamma(r)$ be a function defined 
on $(0, R)$ with $\int_0^R \gamma(r) dr=\infty$. Let $f(z)$ be a meromorphic function on  $\bigtriangleup(R)$. 
Then, for $\delta>0$,  the inequality
$$\int_0^{2\pi} \log^+\left|{f'\over f}(re^{i\theta})\right|{d\theta\over 2\pi} 
\leq  (1+\delta) \log \gamma(r) + \delta \log r + O(\log T_f(r))$$
holds outside a set $E\subset (0, R)$ 
with 
$\int_E  \gamma(r)   dr <\infty$.
\end{theorem}
\begin{proof}  For $w\in {\Bbb C}$, we define the $(1,1)$ form on ${\Bbb C}$ with singularities at $w=0, \infty$:
$$\Phi={1\over (1+\log^2 |w|)|w|^2}{\sqrt{-1}\over 4\pi^2}dw\wedge d\bar{w}.$$

 The form $\Phi$ is of integral 1. By the change of variable formula,
$$\int_{\bigtriangleup(t)} f^*\Phi = \int_{w\in {\Bbb C}}  n_{f}(t, w) \Phi(w).$$
Thus, defining
 $\mu(r):= \int_0^r {dt\over t} \int_{\bigtriangleup(t)} f^*\Phi$, we have
\begin{eqnarray*}
 \mu(r)&=& \int_0^r {dt\over t} \int_{\bigtriangleup(t)} {|f'|^2\over (1+\log^2 |f|)|f|^2}{\sqrt{-1}\over 4\pi^2}dz\wedge d\bar{z}\\
&=& \int_{w\in {\Bbb C}} \int_0^r {dt\over t} n_{f}(t, w) \Phi(w)
=  \int_{w\in {\Bbb C}} N_{f}(r, w) \Phi(w) \leq T_f(r)+O(1)
\end{eqnarray*}
where the last inequality holds as a consequence of  the  First Main Theorem. 
Using the observation (\ref{cs}) (or Lemma \ref{calculus}) we get
$${1\over 2\pi}  \int_{|z|=r}  {|f'|^2\over (1+\log^2 |f|)|f|^2} {d\theta\over 2\pi} 
\leq {1\over 2} r^{2\delta}\cdot  \gamma^{2+2\delta}(r)  \cdot T^{(1+2\delta)^2}_f(r) $$
 outside a set $E\subset (0, 1)$ with $\int_E  \gamma(r)  dr <\infty$.
By making use of this, the Calculus lemma and the concavity of the logarithm function,
we carry out the following classical computations, except for the error term:
\begin{eqnarray*}
\lefteqn{\int_0^{2\pi} \log^+\left|{f'\over f}(re^{i\theta})\right|{d\theta\over 2\pi} }\\
&=& {1\over 2} \int_{|z|=r} \log^+ \left({|f'|^2\over (1+\log^2 |f|)|f|^2}
((1+\log^2 |f|)\right){d\theta\over 2\pi}\\
&\leq &  {1\over 2} \int_{|z|=r} \log^+ \left({|f'|^2\over (1+\log^2 |f|)|f|^2}\right){d\theta\over 2\pi} \\
&~&+ {1\over 2} \int_{|z|=r} \log^+(1+(\log^+|f|+\log^+(1/|f|))^2) {d\theta\over 2\pi}\\
&\leq& {1\over 2} \int_{|z|=r} \log  \left(1+ {|f'|^2\over (1+\log^2 |f|)|f|^2}\right){d\theta\over 2\pi}\\
&~&+ \int_{|z|=r} \log^+(\log^+|f|+\log^+(1/|f|)){d\theta \over 2\pi}+{1\over 2}\log 2\\
&\leq& {1\over 2}\log\left(1+
 \int_{|z|=r}  {|f'|^2\over (1+\log^2 |f|)|f|^2}{d\theta\over 2\pi}\right)\\
&~&+ \int_{|z|=r} \log(1+\log^+|f|+\log^+(1/|f|)){d\theta\over 2\pi} +{1\over 2}\log 2\\
&\leq& {1\over 2}\log\left(1+\pi  r^{2\delta}\cdot  \gamma^{2+2\delta}(r)  \cdot T^{(1+2\delta)^2}_f(r)
 \right)\\
&~&+ \log\left(1+m_f(r, \infty)+m_f(r, 0)\right)+ {1\over 2}\log 2\\
&\leq& {1\over 2} \log \left(1+\pi r^{2\delta}\cdot  \gamma^{2+2\delta}(r) \cdot T^{(1+2\delta)^2}_f(r)\right)
+ \log^+T_f(r) + O(1)\\
&\leq&  (1+\delta) \log \gamma(r)+ \delta \log r + O(\log T_f(r))
\end{eqnarray*}
holds outside a set $E\subset (0, R)$ 
with 
$\int_E  \gamma(r) dr <\infty$.
This proves the theorem.
\end{proof}
We actually need to estimate the higher order 
derivatives.
\begin{theorem}\label{l2} Let $0<R\leq \infty$ and  let $\gamma(r)$ be a function defined 
on $(0, R)$ with $\int_0^R \gamma(r) dr=\infty$. 
Let $f(z)$ be a meromorphic function on  $\bigtriangleup(R)$. 
Then for $k\ge 1$  and $\delta>0$ (small enough),  the inequality
\begin{eqnarray*}
\int_0^{2\pi} \log^+\left|{f^{(k)}\over f}(re^{i\theta})\right|{d\theta\over 2\pi} 
&\leq& (1+\delta)k \log \gamma(r) +\delta k\log r\\
&~&+ O(\log T_f(r)+\log\log  \gamma(r)+\log\log r)
\end{eqnarray*}
holds outside a set $E\subset (0, R)$ 
with 
$\int_E \gamma(r) dr <\infty$.
\end{theorem}

\begin{proof} Note that 
$${f^{(k)}\over f} = {f^{(k)}\over f^{(k-1)}}{f^{(k-1)}\over f^{(k-2)}}\dots {f'\over f}$$
hence, by using  Theorem \ref{l}, 
$$\int_0^{2\pi} \log^+\left|{f^{(k)}\over f}(re^{i\theta})\right|{d\theta\over 2\pi} 
\leq \sum_{j=1}^k \log^+\left|{f^{(j)}\over f^{(j-1)}}(re^{i\theta})\right|{d\theta\over 2\pi}$$
$$\leq (1+\delta) k\log \gamma(r) + \delta k\log r + O\left(\sum_{j=1}^k \log T_{f^{(j-1)}}(r)\right)$$
holds outside a set $E\subset (0, R)$ 
with 
$\int_E  \gamma(r) dr <\infty$.
On the other hand, 
\begin{eqnarray*}
&~&T_{f^{(j-1)}}(r) = m_{f^{(j-1)}}(r, \infty) + N_{f^{(j-1)}}(r, \infty)\\
&\leq& m_{f^{(j-1)}/f^{(j-2)}}(r, \infty) + m_{f^{(j-2)}}(r, \infty) + 2T_{f^{(j-2)} }(r)+O(1)\\
&\leq&\int_0^{2\pi} \log^+\left|{f^{(j-1)}(re^{i\theta})\over 
f^{(j-2)}(re^{i\theta})}\right| {d\theta\over 2\pi} + 2T_{f^{(j-2)}}(r) + O(1)\\
&\leq&  (1+\delta)\log \gamma(r) + \delta \log r + O( \log T_{f^{(j-2)}}(r)) +  2T_{f^{(j-2)}}(r)
\end{eqnarray*}
holds outside a set $E\subset (0, R)$ 
with 
$\int_E  \gamma(r) dr <\infty$.
 The theorem is proved by induction.
\end{proof}
We now extend the above theorem to jet differentials.   Jet bundles are generalizations of tangent bundles. 
Kobayashi attributes the introduction of the concept of jets and jet bundles to Ehresmann. We refer to \cite{GreenG}, Kobayashi's book \cite{Koba} and Demailly's survey paper \cite{Dem}. See also \cite {PS}.
Let $X$ be a complex manifold with $\dim X=n$. 
Let $x\in X$ and consider the germs of 
holomorphic mappings $\phi: \bigtriangleup(1)  \rightarrow X$ with 
$\phi(0)=x$. Two germs $\phi, {\tilde \phi}$ osculate to order $k$  (denote it as $\phi \sim^k   {\tilde \phi}$)
if $\phi^{(i)}(0)=  {\tilde \phi}^{(i)}(0)$, for  ${0 \leq i \leq k}$.
Let $j_k(\phi)$ denote the equivalence class of $\phi$ and set
$$J_k(X)_x=\{j_k(\phi)~|~\phi: ( \bigtriangleup, 0) \rightarrow (X, x)\}.$$
Clearly $J_k(X)_x = {\Bbb C}^{nk}$, i.e. every element $v \in J_k(X)_x$ is represented by 
$({d^j\over d\zeta ^j} (z^i \circ \phi)(0))_{1 \leq  j
 \leq k, 1 \leq  i  \leq  n}$ for some holomorphic 
map $\phi$ from an open neighborhood  $U$  of  $0$  in  
${\Bbb  C}$ to $M$ such that  $\phi(0) = x$. Of course this isomorphism depends on the choice of 
local coordinates $z^1, \dots, z^n$. 
Let $J_k(X)=\cup_{x\in U} J_k(X)_x$. Locally
 $J_k(U)=U\times {\bf C}^{kn}$, so $J_k(X)$ is a complex manifold of dimension $n+nk$. 
 For a holomorphic map  $f: \bigtriangleup(R)  \rightarrow X$,  at each point $z\in \bigtriangleup$, the map 
$f$ has a jet in $J_k(X)_{f(z)}$, denoted by $j_kf(z)$. 
The notation $j_k(f):  \bigtriangleup(R) \rightarrow J_k(X)$ will be used to denote the
natural lifting of $f$ to $k$-jet. 
The $1$-jet bundle  $J_1(X)$  is simply the tangent bundle of  $M.$ 
For $k>1$, $J_k(X)$ is no longer a vector bundle, just a {\it holomorphic fiber bundle}, i.e.
$J_k(X)$ is a complex analytic space with a natural projection $p: J_k(X)\rightarrow X$ with 
$p^{-1}(U)=U\times {\Bbb C}^{nk}$.

When $X$ is an analytic set, we can consider the space $J_k(RegX).$ Let $G_k$ denote the 
group of $k-$jets of biholomorphisms of $({\Bbb C},0)$. One can consider the space
$J_k(RegX)/G_k$ following \cite{Dem}, one can construct a compactification $X_k$ of this space. There is a natural projection $\pi_k: X_k\to X$ , the fiber at a non-singular point is a rational manifold.  See  \cite{Dem} for more details.

Let $x\in X$ and let $z^1, \dots, z^n$ be a local coordinate of $X$ centered at $x$. We consider the symbols $$dz^1, \dots, dz^n, d^2z^1, \dots, d^2z^n, \dots, d^kz^1, \dots, d^kz^n$$
and we say that the weight of the symbol $d^pz^i$ is equal to $p$, for any $i=1, \dots, n$. A (Green-Griffiths) 
 {\it jet differential} of order $k$ and degree $m$ at $x$ is a homogeneous polynomial of weighted degree $m$ in $(d^pz^i)_{p=1, \dots, k, i=1\dots, n}$, when $d^pz_j$ is given the weight $p$. We denote 
$E_X^{k, m}$ the set of (Green-Griffiths)  jet differentials of total weight $m$ and order $k$. 

Let $D=Y_1+\cdots +Y_l$ be an effective
divisor, such that the pair $(X, D)$ is log-smooth (this last condition
means that the hypersurfaces $Y_j$ are non-singular, and that they have
transverse intersections). 
A {\it jet differential} of order $k$ and degree $m$ with possible log-pole along $D$ is locally  a homogeneous polynomial of weighted degree $m$ in $d^p\log z^1, \dots, d^p\log z^d, d^pz^{d+1}, \dots, d^pz^n$
where $p=1, \dots, k$ and $z^1\cdots z^d=0$ is a local defining equation of the divisor $D$.
We denote $E_X^{k,m}(\log D)$ the set of  jet differential of order $k$ and degree $m$ with possible log-pole along $D$.  

The  Logarithmic Derivative Lemma is extended to the jet differentials with possible log-pole along $D$ as follows.
\begin{theorem}[Logarithmic derivative lemma for jet differentials]\label{log} Let $X$ be a complex projective  manifold and let $D$ be a  divisor on $X$  such that the pair  $(X, D)$ is log-smooth.  Let $A$ be an ample divisor on $X$ and 
$\omega_A$ be its curvature form.
Let $\mathcal P$ be a logarithmic $k$-jet differential along $D$ on $X$ (of degree $m$).
Let  $f: \bigtriangleup(R)\rightarrow X$ be a holomorphic map such that $f( \bigtriangleup(R))\not\subset D$. Let $\xi(z):=\mathcal P(J_k(f))(z)$ which is a meromorphic function on  $\bigtriangleup(R)$.
Assume that $c_{f, \omega_A}<\infty$. 
 Then, for $\epsilon>0$,  the inequality 
$$\int_{0}^{2\pi} \log^+ |\xi(re^{i\theta})|
\frac{d\theta}{2\pi}
\leq C( (c_{f, \omega_A}+\epsilon)T_{f, A}(r)
      +\epsilon \log r+ \log T_{f, A}(r))$$
 holds outside a set $E\subset (0, R)$ 
with $\int_E  e^{(c_{f, \omega_A}+\epsilon)T_{f, A}(r)}  dr <\infty$, where $C>0$ is a constant.
 \end{theorem}
\begin{proof}  We follow the argument in \cite{SY2} (see also  \cite{Ru}, Theorem A7.5.4).
Since $X$ is projective, 
we can embed $X$ into a projective space ${\Bbb P}^N$ with homogeneous coordinates $[w_0:\cdots:w_N].$
Let $Z=\{\prod_{i=0}^N w_i=0\} \subset  {\Bbb P}^N$. Choose elements ${\tilde A}_t \in GL(N + 1,{\Bbb C})$, $0\leq t\leq N$
 such that $\cap_{t=0}^N A_t(Z)=\emptyset$, where $A_t: {\Bbb P}^N\rightarrow {\Bbb P}^N$  
 is the map induced by $\tilde A_t$.  Let
$$\{u_{j, \nu}\}_{0\leq j\leq N, 1\leq \nu \leq N(N+1)}:= \left\{{w_{\lambda}\over w_j}\circ A_t\right\}_{0\leq \lambda\leq N, \lambda\not=j, 0\leq t\leq N}.$$
Then for any point $P_0\in {\Bbb P}^N$ there exist $0\leq j_1, \dots, j_N \leq N, 1\leq \nu_1, \dots, \nu_N\leq N(N+1)$, 
such that one can choose local branches $\log u_{j_1, \nu_1}, \dots, \log u_{j_N, \nu_N}$
to form a local coordinate system of ${\Bbb P}^N$ at $P_0$.
 As a consequence there exists a positive constant $C$ such that 
$$ |f^*{\mathcal P}|\leq  C\sum_{j=0}^N \sum  \left|f^* \prod_{\nu=0}^{N(N+1)} (d^{\alpha_{j, \nu}} \log u_{j, \nu})^{\beta_{j, \nu}}\right|,$$
where the second summation $\sum$ is over the indices $\{\alpha_{j, \nu},   \beta_{j, \nu}\}_{1\leq \nu\leq N(N+1)}$, 
 with $\sum_{\nu=1}^{N(N+1)} \alpha_{j, \nu}   \beta_{j, \nu}=m$, $0\leq \alpha_{j, \nu}\leq k,    \beta_{j, \nu} \ge 0.$
 Since $f^*{\mathcal P}= \xi (d\zeta)^m$,  the above gives 
 $$\int_0^{2\pi} \log^+ |\xi(re^{i\theta})| {d\theta\over 2\pi} 
\leq  C' \sum_{h \in {\mathcal H}} \sum_{1 \leq s \leq k}\int_0^{2\pi} 
\log^+ \left|{(h\circ f)^{(s)}\over h\circ f}(re^{i\theta})\right|{d\theta\over 2\pi},$$
where $C' > 0$ is a constant, and ${\mathcal H}$ is the set $\{u_{j, \nu}\}$.
By applying Theorem \ref{l2} with $\gamma(r): = \exp((c_{f, \omega_A}+\epsilon)T_{f, A}(r))$, the inequality
\begin{eqnarray*}
 \int_{0}^{2\pi} \log^+ \left|{(h\circ f)^{(s)}\over h\circ f}(re^{i\theta})\right|
\frac{d\theta}{2\pi}
  &\leq&   (1+\epsilon)s (c_{f, \omega_A}+\epsilon)T_{f, A}(r)
      +\epsilon s\log r\\
&~&+ O(\log T_{h\circ f}(r)) \end{eqnarray*}
      holds outside a set $E\subset (0, R)$ 
with $\int_E  e^{(c_{f, \omega_A}+\epsilon)T_{f, A}(r)}  dr <\infty$.
Since $h$ is a rational function,
$$ \log T_{h \circ f}(r) \leq O(\log T_{f, A} (r)) $$
and we arrive at the estimate
\begin{eqnarray*}
\int_{0}^{2\pi} \log^+ |\xi(re^{i\theta})|
\frac{d\theta}{2\pi}
  &\leq& C( (c_{f, \omega_A}+\epsilon)T_{f, A}(r)
      +\epsilon \log r+ \log T_{f, A}(r)),
\end{eqnarray*}
for some constant $C>0$, where the inequality holds outside a set $E\subset (0, R)$ 
with $\int_E  e^{(c_{f, \omega_A}+\epsilon)T_{f, A}(r)}  dr <\infty$.
\end{proof}
As a corollary of the above Theorem, we get the following result. 
\begin{corollary}[Fundamental Vanishing Theorem]\label{f}
Let $X$ be a complex projective  manifold.
Let $f: \bigtriangleup(R)\rightarrow X$ be a  holomorphic map.  
Assume that  $f\in {\mathcal E}_0$, i.e  $\int_0^R exp(\epsilon T_{f, A}(r)) dr=\infty$ for any $\epsilon>0$ for some $($hence for any$)$ ample divisor $A$.
  Let   $\mathcal P$ be a holomorphic (or log-pole) $k$-jet differential (of degree $m$) on $X$ which
vanishes on an ample divisor $A$ of $X$ $($and the image of $f$ is disjoint from the
log-pole of $\mathcal P$$)$, i.e. $\mathcal P\in H^0(X, E_X^{k,m}\otimes {\mathcal O}(-A))$ or 
$\mathcal P\in H^0(X, E_X^{k,m}(\log D)\otimes {\mathcal O}(-A))$. Then $f^*\mathcal P$ is identically zero on $\bigtriangleup(R)$.
\end{corollary}

\noindent{\bf Remark}. We observe that  if $R=\infty$,  then $f$ is necessarily in ${\mathcal E}_0$ if $f$ is non-constant.
So the above result extends the  Fundamental Vanishing Theorem for maps defined in the complex plane ${\Bbb C}$.  See Green-Griffiths \cite{GreenG}, Siu-Yeung \cite{SY2} and Demailly's survey paper \cite{Dem}.
\begin{proof}  Assume that $f^*\mathcal P\not\equiv 0$, we will derive a contradiction.
 Choose a positive integer $l$ such that $lA$ is very ample. 
The canonical map $\phi_{lA}$ associated to $lA$ embeds $X$ into the  projective space  ${\Bbb P}^N({\Bbb C})$ with homogeneous 
coordinates $[w_0: \dots: w_N]$. By Cartan's Second Main Theorem, we conclude that for any $0< \epsilon <1$, 
there exists a hyperplane $H = \{[w_0: \cdots: w_N]~|~ \sum_{i=0}^N a_i w_i = 0\}$ such that 
$$N_{\phi_{lA}\circ f}(r, H) \ge (1-\epsilon)T_{\phi_{lA}\circ f}(r).$$
Let $s_A$ denote the canonical section of  of the line bundle associated to  $A$ (i.e. 
$[s_A=0]=A$). By replacing $\mathcal P$ by 
$\left({{\mathcal P}\over s_A}\right)^l\phi_{lA}^*(\sum_{i=0}^N a_i w_i)$
we can assume without loss of generality that $\ell = 1$ and  $A = \phi_{lA}^* H$ so we have 
\begin{equation}\label{m1}
N_f(r, A) \ge (1-\epsilon)T_{f, A}(r).
\end{equation}
Write $f^*\mathcal P(z) = \xi (dz)^{\otimes m}$.
Since $\mathcal P$ vanishes on $A$,  by  (\ref{m1}),  the Jensen formula and Theorem \ref{log} (noticing that $c_{f, \omega_A}=0$ under our assumption),
\begin{eqnarray*}
(1-\epsilon) T_{f, A}(r) &\leq& N_f(r, A) \leq \int_0^{2\pi} \log |\xi(re^{i\theta})| {d\theta\over 2\pi} \\
&\leq&  C (\epsilon T_{f, A}(r)
      +2\epsilon \log r+ \log T_{f, A}(r))
\end{eqnarray*}
 holds outside a set $E\subset (0, R)$ 
with $\int_E  e^{\epsilon T_{f, A}(r)}  dr <\infty$, which gives a 
contradiction by taking $\epsilon$ small enough.
\end{proof}
\section{Bloch's theorem and the Second Main Theorem for mappings into Abelian varieties}

\noindent{\bf A. Bloch Theorem}.

The following is a fundamental theorem in value distribution theory (see Bloch \cite{Bloch}, Siu \cite{Siu}, Noguchi-Ochiai \cite{NO}, and \cite{PS}, \cite{Ru}).
\begin{theorem}[Bloch] Let $A$ be an Abelian variety and let
  $f: {\Bbb C} \rightarrow A$ be a holomorphic map. Then the 
Zariski closure of $f({\Bbb C})$ is a translate of a sub-abelian variety. 
\end{theorem}
We extend the above result to mappings on the disc. We follow the strategy from Siu \cite{Siu}
 as carried out in \cite{PS} where ${\Bbb C}$ is replaced by a parabolic Riemann Surface.
We recall the following result due to Ueno \cite{Ueno}.
\begin{theorem}[Ueno]\label{ueno} Let $X$ be a subvariety of a complex torus $T$. Then
there exist a complex torus $T_1\subset T$, a projective variety $W$ and an
abelian variety $A$ such that

(1) We have $W\subset A$ and $W$ is a variety of general type;

(2) There exists a dominant (reduction) map ${\mathcal R}: X\rightarrow W$ whose
general fiber is isomorphic to $T_1$.
\end{theorem}
We now prove the following result.
\begin{theorem}\label{bloch}
Let $T$ be a complex torus and let  $f: \bigtriangleup(R) \rightarrow T$ be a non-constant holomorphic map in the space ${\mathcal E}_0$ $($i.e  $\int_0^R exp(\epsilon T_f(r)) dr=\infty$ for any $\epsilon>0$$)$.
Let $X$ be the Zariski closure of $f(\bigtriangleup(R))$. Then either 
 $X$  is the translate of a sub-torus of $T$, or
there is a variety of general type $W$ and map $\mathcal R: X\rightarrow W$ 
such that $\mathcal R\circ f$ does not belong to the space ${\mathcal E}_0.$
\end{theorem}

\noindent{\bf Remarks}.  (1)  The characteristic function $T_f(r)$ is defined by $T_f(r)=T_{f, \omega}(r)$ where 
$\omega=\pi_*(dw_1+\cdots +dw_m)$ where $\pi: {\Bbb C}^m\rightarrow T$ is the projection map.
(2) We observe that  if $R=\infty$,  then $f$ and $\mathcal R\circ f$ are necessarily in ${\mathcal E}_0.$
So the above result extends the classical Bloch's Theorem.

To prove Theorem \ref{bloch}, 
 let $n$  be the complex dimension of  $T$.  Let $J_k(T) =T\times  {\Bbb C}^{kn},$ and 
 $J^c_k(T) =T\times  {\Bbb P}^{nk-1}$.
Let $\mathcal X_k$ be the Zariski closure of  $j_k(f)(\bigtriangleup(R))$  in 
$J^c_k(T).$ 
Let $\tau_k : \mathcal X_k \rightarrow {\Bbb P}^{nk-1}$ be  the projection on
the second factor.
The proof relies on the following two Propositions whose idea goes back to Bloch \cite{Bloch} (see
also \cite{Dem} and \cite{PS}).
\begin{proposition}[See Proposition 5.3 in \cite{PS}]\label{a} Assume that the Zariski closure of $f$ is $X$.
We assume that for each $k \ge 1$  the fibers of $\tau_k$ are
positive dimensional. Then the dimension of the subgroup $A_X$ of $T$
defined by
$$A_X := \{a \in  T ~|~ a + X = X\}$$
is strictly positive.
\end{proposition}

In the following statement we discuss the other possibility.
\begin{proposition}\label{b} Let $k$ be a positive integer such that the map
$\tau_k: \mathcal X_k\rightarrow  {\Bbb P}^{nk-1}$
has finite generic fibers. Then there exists a jet differential $\mathcal P$ of order
$k$ with values in the dual of an ample line bundle, and whose restriction
to $\mathcal X_k$ is non-identically zero.
\end{proposition}
\begin{proof} The hyperplane line bundle $\mathcal O_{{\Bbb P}^{nk-1}}(1)$
is ample, and since the generic fibers of $\tau_k$ are of dimension zero, the restriction to $\mathcal X_k$ of
the line bundle $\mathcal O_k(1): =\tau_k^* \mathcal O_{{\Bbb P}^{nk-1}}(1)$ is big. Hence, for $m>>0$  large enough, we have
$$H^0(\mathcal X_k, \mathcal O_k(m)\otimes A^{-1})\not =\emptyset,$$
which means that there exists a jet differential $\mathcal P$ of order
$k$ with values in the dual of an ample line bundle $A$, and whose restriction
to $\mathcal X_k$ is non-identically zero.
The proposition is proved.
\end{proof}

\noindent{\it Proof of Theorem \ref{bloch}}. 
  Let $X$ be the Zariski closure of $f$. 
Thanks to Ueno's result (Theorem \ref{ueno}),  we can
consider the reduction map $\mathcal R: X \rightarrow W$. We claim
that,  if 
$X$ is not a translate of a sub-torus, 
 then   $\mathcal R\circ f$  is not in the space ${\mathcal E}_0.$ If W is a point, this means that $X$ is the translate of a sub-torus.
If this is not the case, then we can assume that $X$ is of general type and $\mathcal R\circ f$ is in
${\mathcal E}_0.$
 If the hypothesis in  Proposition \ref{b}  is verified, then $ \mathcal X_k$ is algebraic and Corollary \ref{f}  gives a contradiction. So
the hypothesis of Proposition \ref{b} will never be verified for any $k\ge 1$.
Hence the hypothesis of the Proposition \ref{a} are verified, and so $X$
will be invariant by a positive dimensional sub-torus of $T$. Since $X$ is
assumed to be a manifold of general type, its automorphism group is
finite, so this cannot happen. 
This finishes the proof.

\noindent{\bf  B. The Second Main Theorem for Holomorphic Curves Into Abelian Varieties}. 

We prove the following result which generalizes the result of Siu-Yeung \cite{SY2} (see also \cite{NWY}, \cite{NW}). 
\begin{theorem} 
Let $A$ be an Abelian variety, and let $D$ be an ample  divisor on 
$A$. Let 
 $f: \bigtriangleup(R)\rightarrow A$ be a holomorphic map with Zariski dense image. Assume that $f\in {\mathcal E}_0$. Then there is a positive integer $k_0$ such that, for any $\epsilon>0$, 
 $$T_{f, D}(r)  \leq N_f^{(k_0)}(r, D) + \epsilon T_{f, D}(r)+ O(\log T_{f, D}(r))+\epsilon \log r $$
holds for $r\in (0, R)$ except for a set $E$ with $\int_E \exp(\epsilon T_{f, D}(r)) dr<\infty$.
  \end{theorem}
When $R=\infty$ then  $f\in {\mathcal E}_0.$ So the above theorem recovers the result of  Siu-Yeung \cite{SY2}. Note that in the case 
$R=\infty$,  K. Yamanoi \cite{Y} showed that one can indeed take $k_0=1$.  The proof here follows from the argument in 
the book by   Noguchi and Winkelmann (see Theorem 6.3.1 in \cite{NW}).

\begin{proof}
For $k\ge 1$, 
 let $X_k(f)$ be the Zariski closure of the image of the $k$-jet lifting $j_k(f)$ of 
$f$. Let $I_k$ denote the restriction to $X_k(f)$  of the jet projection 
$p_k: J_k(A)=A\times {\Bbb C}^{nk}\rightarrow  {\Bbb C}^{nk}$, where $n=\dim A$.
Let $x\in D$ and $\sigma=0$ be a local defining equation of $D$ near $x$. 
For a given holomorphic map $\phi: (\bigtriangleup(1), 0)
\rightarrow (A, x)$,  we denote its $k$-jet by $j_k(\phi)$ and write
$$d^j\sigma(\phi)={d^j\over d\zeta^j}|_{\zeta=0} \sigma(\phi(\zeta)).$$
We set $J_{k, x}(D)=\{j_k(\phi)\in J_k(A)~|~ d^j\sigma(\phi)=0, 1 \leq j \leq k\}$, and 
$J_k(D)=\cup_{x\in D} J_{k, x}(D)$.
To continue the proof, we need the following key lemma.

\noindent{\bf Key Lemma}. {\it There is $k_0\in {\Bbb N}$ such that for $k\ge k_0$
$$I_k(X_k(f))\cap I_k(J_k(D))\not=I_k(X_k(f)).$$}

\noindent{\it Proof}. It suffices to show that there is $k\in {\Bbb N}$ such that
$I_k(j_k(f)(0))\not\in I_k(J_k(D)).$
Suppose that $I_k(j_k(f)(0))\in I_k(J_k(D))$ for all integers $k\ge 0$.
 Then we have that 
$$J_k(D)\cap I_k^{-1}(I_k(j_k(f)(0)))\not=\emptyset$$
for all $k\ge 0$. 
Define $$V_k: =p_{1, k}(J_k(D)\cap I_k^{-1}(I_k(j_k(f)(0))))\not=\emptyset,$$
where $p_{1, k}$ is the projective $J_k(A)\rightarrow A$. 
Note that $V_k$ is Zariski closed (because $p_{1, k}: J_k(A)\rightarrow A$ has a section 
$id_A\times \{I_k(j_k(f)(0))\}: A \rightarrow J_k(A)$, and $V_k$ is the pull-back of supp$J_k(D)$
by this section), and note that $V_{k+1}\subset V_k$. Thus we have the sequence of Zariski closed set
$$\cdots\subset V_3\subset V_2 \subset V_1\subset D$$
that eventually stabilizes at the variety $V$. Since we are assuming that $V_k\not=\emptyset$, $V$ is not empty. 
Let $a\in V$, and translate $f$ by $a-f(0)$, i.e
${\tilde f}(z)=f(z)+a-f(0)$. Then by the construction of ${\tilde f}$, we have
 ${\tilde f}(0)=a$ and $j_k({\tilde f})(0) \in J^k(D)$. 
Considering the Taylor series, we get ${\tilde f}(\bigtriangleup(R))
\subset D$, and hence a contradiction since we are assuming that $f$ is non-degenerate.
Thus the lemma is proved.

Write  $Y_k:=I_k(X_k(f))$. Note that 
$I_k$ is proper, therefore $Y_k$ is an irreducible algebraic subset of ${\Bbb C}^{nk}$. 
By the key lemma, there is $k=k_0$ for which there is a polynomial $P$ on ${\Bbb C}^{nk}$ satisfying 
$$P|_{Y_k}\not\equiv 0, ~~~P|_{J_k(D)}\equiv 0.$$
Let $\{U_{\lambda}\}$ be an affine covering of $A$ such that $D\cap U_{\lambda}=\{\sigma_{\lambda}=0\}$
for a regular function $\sigma_{\lambda}$ on $U_{\lambda}$. The defining functions of 
$J_k(D)|_{U_{\lambda}}$ are given by 
$$\sigma_{\lambda}=d\sigma_{\lambda}=\cdots = d^k\sigma_{\lambda}=0.$$
On each $U_{\lambda}$ one obtains the following equation:
$$a_{\lambda 0}\sigma_{\lambda} + \cdots + a_{\lambda k}d^k\sigma_{\lambda}=I_k^*P|_{U_{\lambda}}.$$
Here $ a_{\lambda j}$ are polynomials in jet coordinates with
coefficients of rational holomorphic functions on $U_{\lambda}$ restricted on $J_k(A)|_{U_{\lambda}}$. 

Using a Hermitian metric on the line bundle $[D]$ associated to  $D$, we have positive
functions $\rho_{\lambda}\in C^{\infty}(U_{\lambda})$ such that ${|\sigma_{\lambda}|\over 
\rho_{\lambda}}= {|\sigma_{\mu}|\over 
\rho_{\mu}}$ on $U_{\lambda}\cap U_{\mu}$. Therefore
$$
\rho_{\lambda} a_{\lambda 0}  + \rho_{\lambda} a_{\lambda 1}{d\sigma_{\lambda}\over \sigma_{\lambda}}+\cdots+
\rho_{\lambda} a_{\lambda k}{d^k\sigma_{\lambda}\over \sigma_{\lambda}}={ \rho_{\lambda}\over 
\sigma_{\lambda}}I_k^*P|_{U_{\lambda}}.$$
Substituting $j_k(f)(z)$, $f(z)\in U_{\lambda}$ in the above equation, we have
\begin{eqnarray}\label{truncate}
&~&\left|
\rho_{\lambda}(f(z)) a_{\lambda 0}(f(z)) + \cdots + 
\rho_{\lambda}(f(z)) a_{\lambda k}(f(z)){{d^k\over dz^k}
\sigma_{\lambda}(f(z))\over \sigma_{\lambda}(f(z))}\right|\nonumber \\
&~&={ |\rho_{\lambda}(f(z))|\over 
|\sigma_{\lambda}(f(z))|}|P(I_k(J_k(f)(z)))|.
\end{eqnarray}
Let $\{\tau_{\lambda}\}$ be a partition of unity subordinated to the 
covering $\{U_{\lambda}\}$. Then 
\begin{eqnarray*} 
&~&{1\over \|\sigma(f(z))\|}\leq {1\over |P(I_k(J_k(f)(z)))|}\\
&~&\times \sum_{\lambda} \left\{\tau_{\lambda}\rho_{\lambda}
|a_{\lambda 0}| + \cdots +  \tau_{\lambda}  \rho_{\lambda}
 |a_{\lambda k}| \left|{{d^k\over dz^k}
\sigma_{\lambda}(f(z))\over \sigma_{\lambda}(f(z))}\right|\right\}.
\end{eqnarray*}
Since $a_{\lambda j}$ are polynomials in jet coordinates with coefficients of
holomorphic functions on $U_{\lambda}$, Theorem \ref{log}  with $\epsilon$ properly chosen yields that
\begin{eqnarray*}
m_f(r, D)&\leq&   C\left( m_{1/ P(I_k(J_k(f)))}(r, \infty)+ \sum_{\lambda, 1 \leq j \leq k}m_{{(\sigma_{\lambda}\circ f)^{(j)}\over \sigma_{\lambda}\circ f}}(r, \infty)\right)\\
&~&+  \epsilon(T_{f, D}(r)+ \log r) + O(\log T_{f, D}(r))
\end{eqnarray*}
holds for $r\in (0, R)$ except a set $E$ with $\int_E \exp(\epsilon T_{f, D}(r)) dr<\infty$, where
 $C>0$ is a constant.
Since $\sigma_{\lambda}$ is a rational function on $A$, $d^j\sigma_{\lambda}/\sigma_{\lambda}$ 
is a logarithmic jet differential carrying logarithmic poles on zeros and poles of 
$\sigma_{\lambda}$. It follows, from Theorem \ref{log} with  $\epsilon$ properly chosen (notice that $c_{f, \omega_D}$ is arbitrarily small in our case),
$$
m_{{(\sigma_{\lambda}\circ f)^{(j)}\over \sigma_{\lambda}\circ f}}(r, \infty)
\leq  \epsilon(T_{f, D}(r)+ \log r) + O(\log T_{f, D}(r))
$$
holds for $r\in (0, R)$ except a set $E$ with $\int_E \exp(\epsilon T_{f, D}(r)) dr<\infty$.
Moreover the First Main Theorem and Theorem \ref{log}  with $\epsilon$ properly chosen imply that 
$$m_{1/ P(I_k(J_k(f)))}(r, \infty)
\leq T_{P(I_k(J_k(f)))}(r)+O(1) \leq  \epsilon (T_{f, D}(r) + \log r)+ O(\log T_{f, D}(r)) $$
holds for $r\in (0, R)$ except a set $E$ with $\int_E \exp(\epsilon T_{f, D}(r)) dr<\infty$.
Thus 
\begin{equation}\label{appro}
m_f(r, D)\leq  \epsilon(T_{f, D}(r) + \log r)+ O(\log T_{f, D}(r))
\end{equation}
holds for $r\in (0, R)$ except a set $E$ with $\int_E \exp(\epsilon T_{f, D}(r)) dr<\infty$.
It is inferred from Theorem \ref{log} with  $\epsilon$ properly chosen   and (\ref{truncate}) that 
\begin{eqnarray*}
N_f(r, D) -N_f^{(k)}(r, D) &\leq& N_{P(I_k(J_k(f)))}(r, 0)\leq T_{P(I_k(J_k(f)))}(r)+O(1)\\
&\leq&  \epsilon T_{f, D}(r) +\epsilon \log r+
O(\log T_{f, D}(r))
\end{eqnarray*}
holds for $r\in (0, R)$ except a set $E$ with $\int_E \exp(\epsilon T_{f, D}(r)) dr<\infty$.
Hence, from the First Main Theorem and (\ref{appro}),
$$T_{f, D}(r) = N_f(r, D) + m_f(r, D) \leq N_f^{(k)}(r, D) +   2\epsilon (T_{f, D}(r) + \log r)+O(\log T_{f, D}(r))$$
holds for $r\in (0, R)$ except a set $E$ with $\int_E \exp(\epsilon T_{f, D}(r)) dr<\infty$.
This finishes the proof.
\end{proof}

\end{document}